\documentclass[
a4paper,
]{scrartcl}

\usepackage[
final,
biblatex,
theoremdefs,
]{allergy}

\usepackage{unixode}
\usepackage{helvetic}

\usepackage{enumitem}

\usepackage{mdframed}

\usepackage{booktabs}
\usepackage{pdflscape}
\usepackage{afterpage}
\usepackage{capt-of}
\usepackage{colortbl}

\usepackage{subcaption}

\usetikzlibrary{calc}
\usetikzlibrary{matrix}
\usetikzlibrary{trees}
\usetikzlibrary{arrows}
\usetikzlibrary{decorations.markings}

\tikzset{commdiag/.style={matrix of math nodes, row sep=3em, column sep=2.5em, text height=1.5ex, text depth=0.25ex,ampersand replacement=\&},
exseq/.style={commdiag, column sep=2em},
diagequal/.style={double, double distance=2pt, -},
>=stealth,
ineq/.style={baseline=(current  bounding  box.center)}, 
}

\tikzset{
	norient/.style={very thick},
	orient/.style={norient, 
		postaction={decorate},
        decoration={markings,mark=at position .55 with {\arrow[draw,fill=black, line width=.5mm]{latex}}}
			},
		stagetree/.style={node distance=2.8cm, 
		cvertex/.style={shape=circle, draw, font=\sffamily, minimum size=2em},
		vertex/.style={cvertex, very thick, fill=OldLace, },
		initfi/.style={draw=FireBrick},
		intermediate/.style={fill=Gainsboro},
	}
}

\newcommand*\demph[1]{\textbf{\emph{#1}}}

\newcommand*\RR{\mathbf{R}}
\newcommand*\NN{\mathbf{N}}

\newcommand*\Man{\mathcal{M}}
\newcommand*\Tan[1][]{\mathsf{T}_{#1}}

\newcommand*\bundle[2]{#1 \leftarrow #2}
\newcommand*\redsec{\mu}

\newcommand*\liealg[1]{\mathfrak{#1}}
\newcommand*\group[1]{\mathsf{#1}}

\NewDocumentCommand\diff{O{\Man}}{\mathfrak{X}\IfNoValueTF{#1}{(\Man)}{(#1)}}
\NewDocumentCommand\Diff{O{\Man}}{\group{Diff}\IfNoValueTF{#1}{(\Man)}{(#1)}}

\newcommand*\GL[1][d]{\group{GL}(#1)}
\newcommand*\OO[1][d]{\group{O}(#1)}
\newcommand*\gl[1][d]{\liealg{gl}(#1)}
\newcommand*\SO[1][d]{\group{SO}(#1)}
\newcommand*\so[1][d]{\liealg{so}(#1)}
\newcommand*\one{\group{1}}

\newcommand*\Id{\mathsf{Id}}

\newcommand*\symgrp{\group{G}}
\newcommand*\symalg{\liealg{g}}
\newcommand*\isogrp{\group{H}}
\newcommand*\isoalg{\liealg{h}}
\newcommand*\redalg{\liealg{m}}

\newcommand*\act{\cdot}
\newcommand*\mul{\,}

\NewDocumentCommand\lifted{O{+}O{\isoc}O{f}}{#3_{#2}^{#1}}

\newcommand*\gb{\bar{g}}

\newcommand*\affgrp[2]{\begin{bmatrix} #1 & #2\\ 0 & 1 \end{bmatrix}}
\newcommand*\affalg[2]{\begin{bmatrix}#1 & #2 \\ 0 & 0 \end{bmatrix}}

\newcommand*\Cinf{\mathcal{C}^{\infty}}
\newcommand*\Lin{\mathcal{L}}

\newcommand*\initial{\circ}
\newcommand*\final{\bullet}

\newcommand*\motion{\Psi}
\newcommand*\trans{\theta}

\newcommand*\Meth{\Phi}

\newcommand*\supp{\operatorname{supp}}

\newcommand*\inv{^{-1}}
\newcommand*\tp{^{\intercal}}

\newcommand*\origin{\mathsf{o}}

\NewDocumentCommand\Jet{O{\origin}O{k}}{J^{#2}_{#1}\Tan\Man}

\newcommand*\prt{\kappa}

\NewDocumentCommand\skel{}{\Sigma}
\NewDocumentCommand\isoc{O{}}{\varpi_{#1}}
\NewDocumentCommand\conn{}{\omega}

\usepackage[blocks]{authblk}

\title{Integrators on homogeneous spaces: Isotropy choice and connections}
\author{Hans Munthe-Kaas}
\affil{Department of Mathematics, University of Bergen, Norway}
\author{Olivier Verdier}
\affil{Department of Mathematics and Statistics, Umeå University, Sweden}

\begin{document}

\maketitle

\begin{center}
	Communicated by Ernst Hairer
\end{center}

\begin{abstract}
	We consider numerical integrators of ODEs on homogeneous spaces (spheres, affine spaces, hyperbolic spaces).
	Homogeneous spaces are equipped with a built-in symmetry.
	A numerical integrator respects this symmetry if it is equivariant.
	One obtains homogeneous space integrators by combining a Lie group integrator with an isotropy choice.
	We show that equivariant isotropy choices combined with equivariant Lie group integrators produce equivariant homogeneous space integrators.
	Moreover, we show that the RKMK, Crouch--Grossman or commutator-free methods are equivariant.
	To show this, we give a novel description of Lie group integrators in terms of stage trees and motion maps, which unifies the known Lie group integrators.
	We then proceed to study the equivariant isotropy maps of order zero, which we call connections, and show that they can be identified with reductive structures and invariant principal connections.
	We give concrete formulas for connections in standard homogeneous spaces of interest, such as Stiefel, Grassmannian, isospectral, and polar decomposition manifolds.
	Finally, we show that the space of matrices of fixed rank  possesses no connection.
\end{abstract}

\newcommand{\bullsep}{$\cdot$\ }

\begin{description}
	\item[Keywords]
		Homogeneous spaces 
		\bullsep 
		Symmetric spaces
		\bullsep 
		Lie Group integrators
		\bullsep 
		Connection
		\bullsep 
		Runge--Kutta
		\bullsep
		Skeleton
		\bullsep 
		Stiefel Manifold
		\bullsep
		Lax Pair
		\bullsep
		Grassmannian
		\bullsep
		Projective space
		\bullsep
		Polar decomposition
		\bullsep
		Constant rank matrices
\item[Mathematics Subject Classification (2010)] 
		22F30
		\bullsep
		53C30
		\bullsep
		14M15
		\bullsep
		65Pxx
		\bullsep
		37Mxx
\end{description}

\tableofcontents

\section{Introduction}

A homogeneous space is a manifold $\Man$ on which a {Lie group} $\symgrp$ acts transitively.
Homogeneous spaces are ubiquitous in numerical analysis and computer vision.
Examples include Stiefel manifolds and spheres in all dimensions~\cite[\S\,IV.9.1]{HaLuWa06}, Grassmannians and projective spaces~\cite{EdArSm98}, symmetric positive definite matrices used in diffusion tensor computing~\cite{PeFiAy06}, isospectral manifolds used in the context of Lax pairs~\cite{CaIsZa97}, and constant rank matrices used in low-rank approximation~\cite[\S\,IV.9.3]{HaLuWa06}.
We study all these cases in thorough details in \autoref{sec:connexamples} and give further examples of homogeneous space in \autoref{sec:conclusion}.

The most convenient way to construct an integrator on a homogeneous space is, as noted in~\cite{IsMKNoZa00}, to combine a Lie group integrator on $\symgrp$ along with an \emph{isotropy map}.

We now illustrate how isotropy maps arise with the following example.

\subsection{Example: integration on the sphere}

When a differential equation is defined on a sphere, it is natural to use a Lie group integrator on the three dimensional rotation group.
To do this, however, one needs an isotropy map, as explained below.

\subsubsection{Forward Euler on the rotation group}

Consider first the differential equation $R' = \omega(R) \mul R$, where $R\in\SO[3]$ and $\omega \in \so[3]$, i.e., $R$ is a $3\times 3$ rotation matrix ($R R\tp = 1$), and $\omega$ is a $3\times 3$ skew-symmetric  matrix ($\omega\tp = - \omega$).
The Lie group version of the forward Euler method with time step $h$ is defined as (see~\eqref{eq:fEuler})
\begin{align}
	\label{eq:so3fEuler}
	R_1 = \exp\paren[\big]{h \omega(R_0)} \mul R_0
	.
\end{align}
Recall that the exponential $\exp(\omega_0)R_0$ is defined as the solution of the differential equation $R' = \omega_0 R$ at time 1 with initial condition $R_0$.
This is consistent with the interpretation that one ``freezes'' the vector field $R\mapsto\omega(R) R$ at the point $R_0$ to obtain the vector field $R\mapsto\omega(R_0) R$, which is often much easier to solve.

\subsubsection{Integration on a sphere}

Consider now a differential equation on the two dimensional \emph{sphere} $S^2$: \begin{align}
	\label{eq:diffeqsphere}
	x' = f(x)
.
\end{align}
This means that $x\in S^2$, and that $f(x)$ is tangent to the sphere at $x$.
We would like to adapt the forward Euler method on rotations~\eqref{eq:so3fEuler} to solve the differential equation~\eqref{eq:diffeqsphere}.
This is done in the following way.
The vector field $f(x)$ at the point $x$ has to be interpreted as an \emph{infinitesimal rotation}, i.e., it is rewritten as \begin{align}
	\label{eq:sphereconsistency}
	f(x) = \omega(x) \times x
,
\end{align} where $\omega(x) \in \RR^3$.
Note that this is equivalent to considering $\omega$ as a map from the sphere to the space $\so[3]$ of $3\times 3$ skew-symmetric matrices, using the well known map
$v \mapsto \hat{v} \coloneqq \bracket{v\times \cdot}$, sometimes called the hat-map~\cite[\S\,9.2]{AbMaRa1988}.
The integrator stemming from the forward Euler method becomes 
\begin{align}
	\label{eq:fEulersphere}
	x_1 = \exp\paren[\big]{h{\hat{\omega}(x_0)} } \mul x_0
	.
\end{align}

\subsubsection{Ambiguity due to isotropy}

There is, however, a typical ambiguity in choosing $\omega(x)$, as many infinitesimal rotations $\omega(x)$  coincide with $f(x)$ \emph{at the point $x$}.
In other words, the solution $\omega$ of the equation $\omega \times x = f$ is not unique, given a tangent vector $f$ at a point $x$ on the sphere.
The fundamental reason for this is that at any given point $y$ on the sphere, the rotations about $y$ leave the point $y$ invariant.
See also \autoref{fig:motionambiguity} for an illustration in a very similar setting.

\begin{figure}
\captionsetup[subfigure]{justification=centering}
	\centering
	\begin{subfigure}{.49\textwidth}
		\centering
			\begin{tikzpicture}[
		gpt/.style={fill=black},
		xpt/.style={fill=red},
		motion/.style={draw, thick, postaction={decorate}},
		decoration={
			markings,
		mark=at position 0.3 with {\arrow{stealth}},
		mark=at position 0.7 with {\arrow{stealth}},
		}
		]
		\coordinate (x0) at (0,-1);
		\coordinate (x1) at (0,1);
		\path[motion, ultra thick] (x0) -- (x1);
		\path[draw, dashed, color=DarkGray] (2,0) -- (-2,0);
	\newcommand*\drawarcangle[1]{
		\draw[motion] (x0) arc[start angle=-#1, end angle=#1, radius={1/sin(#1)}] ;
		\draw[motion] (x0) arc[start angle={180+#1}, end angle={180-#1}, radius={1/sin(#1)}] ;
		\node[below] at (x0) {$x_0$};
		\node[above] at (x1) {$x_1$};
	}
	\drawarcangle{45}
	\drawarcangle{90}
	\drawarcangle{120}

		\draw[xpt] (x0) circle(2pt);
		\draw[xpt] (x1) circle(2pt);

	\end{tikzpicture}
	\caption{Motion ambiguity
\label{sfig:groupmotions}}
\end{subfigure}
	\begin{subfigure}{.49\textwidth}
		\centering
				\begin{tikzpicture}[
			isochoice/.style={draw, thick,},
			rotated/.style={Green}
			]
		\coordinate (origin) at (0,0);
		\newcommand*\myangle{40}
		\newcommand*\myradius{2}
		\draw[gray, very thick, ->, >=stealth] (origin) -- +(0,1.5) node (tip) {};
		\node[above] at (tip) {$f(x)$};
		\draw[isochoice, ->, >=stealth] (origin) arc[start angle=180, end angle={180-\myangle}, radius=\myradius] node[right] {$\exp(\xi)\act x$};
		\draw[isochoice, rotated, ->, >=stealth] (origin) arc[start angle=0, end angle=-\myangle, radius=\myradius];
		\draw[isochoice, ->, >=stealth] (origin) arc[start angle=180, end angle={180+\myangle}, radius=\myradius];
		\draw[->, >=latex,  rotated] ({180-80}:.5cm) arc[start angle={180-80}, end angle={180+80}, radius=.5];
		\node[red] at (0,-1) {$\neq$};
		\draw[fill=black] (origin) circle(2pt);
		\node[right] at (origin) {$x$};
		\end{tikzpicture}
		\caption{Infinitesimal motion  ambiguity
\label{sfig:infinitesimal}}
	\end{subfigure}
	\caption[Motion Figure]{
		We illustrate the concept of a connection in the case of the Euclidean displacement group $\SO[2] \ltimes \RR^2$ acting on $\RR^2$.

		(\subref{sfig:groupmotions})
		There are plenty of displacements bringing the point $x_0$ to the point $x_1$.
		The ambiguity in the displacement choice is measured by the isotropy group $\SO[2]$, which, in this case, has dimension one.
		The displacements consist of either a translation, or arcs of circle with centre located on the perpendicular bisector of $x_0$ and $x_1$.
		The discovery of Nomizu~\cite{No54} is that there is one \emph{preferred} choice: the translations, marked in bold in the picture.
		The reason translations stand apart, is that they are geodesics associated to an invariant principal connection (see \autoref{sec:princconn}).
		In our framework, they are associated to an equivariant isotropy map, which we call a \emph{connection} (see \autoref{def:connection}).

		(\subref{sfig:infinitesimal})
		Suppose that a choice was made to use circular motions, say, on the right, which gives the point $\exp(\xi)\act x$ on the picture.
		As the isotropy choice is a linear map, the isotropy choice of the opposite of this vector commands to follow $\exp(-\xi)\act x$, i.e., to turn \emph{left}.
		On the other hand, rotating the initial vector by a half turn would command to turn \emph{right}.
		This choice discrepancy illustrates the lack of equivariance.
		In \autoref{prop:affineconnection}, we show rigorously that translations are the only equivariant choices for full affine spaces corresponding to $\GL \ltimes \RR^d$ acting on $\RR^d$.
	}
\label{fig:motionambiguity}
\end{figure}

Resolving that ambiguity amounts to choose a mapping $f \mapsto \omega$, which satisfies the consistency condition~\eqref{eq:sphereconsistency}.
We call such a mapping an \emph{isotropy map}.

Some of the questions we address in this paper are:
\begin{itemize}
	\item
How to choose the isotropy map?
\item
Are some choices better than others?
\item
How to \emph{classify} isotropy maps?
\end{itemize}

\subsubsection{Isotropy forms}

Let us now restrict the discussion on the particular case where the isotropy map $f\mapsto \omega$ depends in fact only on the value of $f$ at each point $x$; we identify such isotropy maps with a $\RR^3$-valued one-form on the sphere and call it an \emph{isotropy form}.
The vector $\omega(x)$ is now determined only from the knowledge of the \emph{vector} $f(x)$.

It is already fairly intuitive that there is a preferred choice, namely the vector $\omega(x)$ which is perpendicular to $x$ and $f(x)$.
A suitable interpretation is that the solution of the \emph{frozen} differential equation $y' = y \mapsto \omega(x) \times y$ starting at $x$ is then an invariant \emph{geodesic} on the sphere.
This reflects a general principle, which is that the isotropy map is actually \emph{equivariant} with respect to the action of the symmetry group $\SO[3]$ (see \autoref{fig:motionambiguity}).
Intuitively, it means in this case that the connection does not make a choice to go left or right, but rather straight ahead, hence the appearance of geodesics.
We show how to identify equivariant isotropy forms with invariant principal connection, and thus, to the existence of geodesics (see \autoref{sec:connections} and \autoref{sec:connexamples}).

\subsubsection{Lifting of vector fields}
One way to look at such an isotropy form is as a computational device to lift a vector field $f$ on the sphere to a vector field $\lifted[][\conn]$ on the group $\SO[3]$.
Let us first choose an arbitrary point on the sphere, say the north pole $N=(0,0,1)$.
We will call this point the \emph{origin} and  denote it by $\origin$ in the rest of the paper.
We now denote by $\SO[2]$ the subgroup of rotation group $\SO[3]$ which leaves the point $N$ invariant.
We will call this subgroup the \emph{isotropy subgroup}.
A point $x\in S^2$ on the sphere may now be identified with the set $R \mul \SO[2]$ of rotations, where $R$ is one of the rotations bringing $N$ to $x$, i.e., $R(N) = x$.
So this set consists of products $R\mul r$ of matrices $R$ with that property, and rotations $r$ which leave $N$ invariant.
One can check that this set is independent of the chosen rotation $R$ bringing $N$ to $x$.
The Lie group $\SO[3]$ can now be regarded as a fibred bundle over the sphere as illustrated on \autoref{fig:homogeneousspace}.
At a point $R\in\SO[3]$ such that $R(N) = x$, the isotropy form $\omega$ defines a lifted vector $\lifted[][\conn](R) \coloneqq \omega(x) \mul R$ which is tangent to $\SO[3]$ at $R$.
The fact that it is a lift means that the projections of the solutions of the equation $R' = \omega(x) R$ on $\Man$ are solutions of $x' = f(x)$.
It may be desirable that the integrator has the same property: we give a precise statement in \autoref{prop:lifting}.
Note however that one does not solve the lifted differential equation, but rather solves directly on the sphere.
Finally, we will mostly focus on \emph{equivariant} isotropy form, which ensures that the described lift is invariant.


\subsubsection{Skeletons}

Note that the version of forward Euler on the sphere is not obtained directly from the forward Euler integrator on $\SO[3]$.
Instead, there is an abstract definition of the forward Euler method which specializes to a concrete version forward Euler on $\SO[3]$ on the one hand, and on the sphere $S^2$ on the other hand.
To describe this mechanism, we introduce the notion of a \emph{skeleton} (see \autoref{sec:rungekutta}), which is, roughly speaking, a Lie group integrator stripped from its isotropy map.
For instance, the skeleton of the forward Euler method is as follows (the meaning of that graph is explained in details in \autoref{sec:rungekutta}):
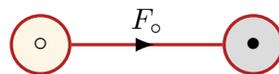
\begin{align}
	\begin{tikzpicture}[stagetree, ineq]
		\node[vertex, initfi] (init) {$\initial$};
		\node[vertex, initfi, intermediate, right of=init] (final) {$\final$};
		\draw[orient, initfi] (init) -- (final) node[midway, auto]{$F_{\initial}$};
	\end{tikzpicture}
\end{align}

As $\SO[3]$ is a particular case of a homogeneous space, one recovers forward Euler on $\SO[3]$ as the forward Euler skeleton applied to a canonical isotropy map, which exists on any Lie group, called the Maurer-Cartan form (see \autoref{prop:lifting}).

One of the results of this paper is that the common skeletons, such as the aforementioned forward Euler, or Runge--Kutta--Munthe-Kaas, or Crouch--Grossman, all are equivariant (\autoref{prop:RKMKCGCFequi}).
We also prove that the combination of an equivariant skeleton with an equivariant isotropy map (not necessarily an isotropy form) yields an equivariant integrator (\autoref{prop:skelequi}).
We refer to~\cite{MLQu01,MKVe13,MLMoMKVe14} for discussions on the importance of affine equivariant integrators and their relations with standard Runge--Kutta methods and B-series.

We now come to the classification of the possible isotropy maps.

\subsection{Classification of isotropy maps}

In this paper we identify five criteria to achieve a classification of isotropy maps: \emph{equivariance}, \emph{locality}, \emph{order}, \emph{flatness} and \emph{symmetry}.
We then recognise the prevalence of equivariant, local, order zero isotropy maps which we call \emph{connections} and study them in details.

Here are the classification criteria in descending order of generality:

\begin{description}
	\item[Equivariance]
The notion of \emph{equivariance}, which was already developed in~\cite{MKVe13} for integrators, is described as follows.
An isotropy map is a map from vector fields to functions from the manifold $\Man$ to the Lie algebra $\symalg$ (see \autoref{sec:isotropymaps}).
As the symmetry group $\symgrp$ acts naturally on both those spaces, an isotropy map is equivariant if it commutes with those actions.
Enforcing equivariance already drastically limits the choices of isotropy maps.
We are not aware of the existence of any non-equivariant isotropy map, although we present some possible candidates in \autoref{sec:conclusion}.

\item[Locality]
Locality, also introduced in a similar guise in~\cite{MKVe13}, means that the isotropy map only depends on the value of the vector field and its derivatives at a given point.
In practice, most isotropy maps are local, but, crucially, the tautological isotropy map (see \autoref{sec:tautological}) is \emph{not} local.

\item[Order]
An equivariant and local isotropy map always has an order, which is the number of derivative of the vector field that the isotropy map depends on~\cite{Ve14}.
Isotropy maps of higher order are possible and have been considered in the literature.
In particular, in~\cite[Example~4]{MK99}, the author considers using the Jacobian in the isotropy map, so this is an isotropy map of order one.
In~\cite{LeOl02}, the authors consider higher order isotropy maps to obtain integrators on the sphere.
When the underlying manifold is an affine space, the integrators using higher order isotropy maps are closely related to \emph{exponential integrators} (see for instance~\cite{Mi05}).
As a detailed study of higher order isotropy maps would take us too far afield, we refer the reader instead to~\cite{Ve14}.

\item[Flatness and Symmetry]

We make a further classification of connections (equivariant, local, order zero, isotropy maps) into \emph{flat} and \emph{symmetric} connections.
A flat connection witnesses the flatness of the ambient homogeneous space, so they occur only in special cases, such as when $\Man$ is an affine space (\autoref{sec:affine}), a principal homogeneous space (an example of which is in \autoref{sec:isoflat}) or a Lie group (\autoref{sec:liegroupact}).
Symmetric connections are, however, more widespread.
There are also a number of connections which are neither flat nor symmetric.
We refer to \autoref{sec:connexamples} and \autoref{tab:homogconn} for examples.
\end{description}



The following results appearing in this paper are new:
\begin{itemize}
	\item
		Classification of isotropy map according to the properties of equivariance, locality and order (\autoref{sec:isotropymaps}, \autoref{sec:connections})
	\item
		Equivalence of connections (order zero equivariant isotropy maps), reductive structures and principal invariant connections in \autoref{sec:connections}
	\item 
		The formulas in \autoref{tab:homogconn}, when they existed in the literature, are in fact connections
	\item
The homogeneous space of fixed rank matrices has no connection (\autoref{prop:fixedranknoconn}).
\item
	New, all encompassing definition of the existing Lie group methods (\autoref{sec:rungekutta}),  showing their equivariance (\autoref{prop:skelequi}) and locality (\autoref{prop:skellocal}).
\item
	If the isotropy choice is equivariant, then so is the corresponding method (\autoref{prop:skelequi}).
\end{itemize}

\subsection{Definitions and notations}

We now recall some basic definitions about homogeneous spaces that will be needed throughout this paper.
The survey paper~\cite{IsMKNoZa00} contains an introduction to manifolds, Lie groups and homogeneous spaces which is targeted at numerical analysts.
The reader will also find in the numerical analysis literature such as~\cite{EdArSm98} and~\cite[\S\,IV.9]{HaLuWa06}, material on some particular examples of homogeneous spaces such as those we study in \autoref{sec:connexamples}.
We refer further to~\cite{Sh97,KoNo69,Ga14} for more thorough expositions of homogeneous spaces.
The reader will find a summary of the concepts and definitions of homogeneous spaces in \autoref{fig:homogeneousspace}.

\begin{figure}
	\centering
	\begin{tikzpicture}[
		scale=1.5,
		fibre/.style={draw, thick},
		helpnode/.style={shape=coordinate},
		liealg/.style={fill=LightGreen, nearly opaque, draw=black},
		proj/.style={dashed, draw},
		vector/.style={->, very thick, draw},
		subalg/.style={draw, thick, color=Green},
		]
		\newcommand*\opos{1}
		\newcommand*\xpos{3}
		\newcommand*\gxpos{4.5}
		\newcommand*\Gleft{0}
		\newcommand*\Gright{5}
		\newcommand*\Ghigh{3}
		\newcommand*\Glow{0}
		\newcommand*\Mmid{-1}

		\NewDocumentCommand\drawvertical{mO{draw}O{}}{
			\path[#2] (#1, \Ghigh) to[bend left=25] #3 (#1, \Glow);
		}
		\NewDocumentCommand\drawhorizontal{mO{draw}}{
			\path[#2] (\Gleft, #1) -- (\Gright, #1);
		}

		\newcommand*\drawpoint[1]{\fill (#1) circle[radius=1pt];}

		\drawhorizontal{\Ghigh}[draw]
		\drawhorizontal{\Glow}[draw]
		\drawhorizontal{\Mmid}[draw]

		\drawvertical{\Gleft}[draw]
		\drawvertical{\Gright}[draw]

		\drawvertical{\opos}[fibre][node[helpnode, pos=.5](e){}]
		\drawvertical{\xpos}[fibre][node[helpnode, pos=.3](gb){} node[helpnode, pos=.5](gbh){}]
		\drawvertical{\gxpos}[fibre][node[helpnode, pos=.4](ggb){} node[helpnode, pos=.6](ggbh){}]

		\newcommand*\liealgsizehalf{.9}
		\newcommand*\liealgsize{1.8}

		\drawpoint{e}
		\node[above left] at (e) {$e$};

		\path[liealg] ($(e) +(\liealgsizehalf, \liealgsizehalf)$) -- ++(0, -\liealgsize) -- ++(-\liealgsize, 0) -- ++(0, \liealgsize) -- cycle;
		\path[subalg] ($(e) +(0, \liealgsizehalf)$) -- +(0, -\liealgsize);
		\path[subalg] ($(e) +(-\liealgsizehalf, 0)$) -- +(\liealgsize, 0);
		
		\node[above left] at ($(e) +(\liealgsizehalf,0)$) {$\redalg$};
		\node[below right] at ($(e) + (0, \liealgsizehalf,0)$) {$\isoalg$};
		\node[below right] at ($(e) + (\liealgsizehalf, \liealgsizehalf,0)$) {$\symalg$};

		\coordinate (o) at (\opos, \Mmid);
		\coordinate (x) at (\xpos, \Mmid);
		\coordinate (gx) at (\gxpos, \Mmid);

		\path[proj] (\opos, \Glow) -- (o);
		\path[proj] (\xpos, \Glow) -- (x);
		\path[proj] (\gxpos, \Glow) -- (gx);

		\node[below] at (\opos, \Mmid) {$\origin$};
		\node[below] at (\xpos, \Mmid) {$x$};
		\node[below] at (\gxpos, \Mmid) {$g \act x$};

		\node[left] at (gb) {$\gb$};
		\node[left] at (gbh) {$\gb \mul h$};
		\node[left] at (ggb) {$g \mul \gb$};
		\node[left] at (ggbh) {$g\mul \gb \mul h$};
		
		\drawpoint{gb}
		\drawpoint{gbh}
		\drawpoint{ggb}
		\drawpoint{ggbh}

		\drawpoint{o}
		\drawpoint{x}
		\drawpoint{gx}

		\node[above] at (\opos, \Ghigh) {$\isogrp$};
		\node[above] at (\xpos, \Ghigh) {$\gb \mul \isogrp$};
		\node[above] at (\gxpos, \Ghigh) {$g \mul \gb \mul \isogrp$};

		\node[right] at (\Gright, \Ghigh) {$\symgrp$};
		\node[right] at (\Gright, \Mmid) {$\Man$};

		\newcommand*\veclength{.5}

		\path[vector] (gb) -- ++(\veclength,.2) node[right] {$\xi \mul \gb$};
		\path[vector] (x) -- +(\veclength,0) node[above] {$\xi \act x$};

	\end{tikzpicture}
	\caption[Homogeneous Space]{
		An illustration of a homogeneous space.
		This illustration's purpose is only to refresh knowledge on homogeneous spaces, the reader not familiar with homogeneous spaces is directed to~\cite{IsMKNoZa00} for an in-depth introduction homogeneous spaces relevant to numerical analysis.

		The Lie group $\symgrp$ acts transitively on the manifold $\Man$.
		The identity of $\symgrp$ is denoted by $e$.
		The tangent space of $\symgrp$ at $e$ is its Lie algebra $\symalg$.

		The origin $\origin\in\Man$ is an arbitrary point of $\Man$ which defines the corresponding isotropy group $\isogrp$.
		The subspace of $\symalg$ which is tangent to $\isogrp$ at $e$ is the Lie algebra $\isoalg$ of $\isogrp$.
		In some cases, as we discuss in \autoref{sec:reductive}, there exists a decomposition $\symalg = \isoalg \oplus \redalg$ such that $[\isoalg,\redalg] \subset \redalg$, which is fundamental to construct isotropy maps.

		Any point $x \in \Man$ may be identified with the set $\gb \mul \isogrp$, for some $\gb$ such that $x = [\gb] = \gb \mul \isogrp$.
		Multiplication on the right by an element $h\in\isogrp$ projects to the same point, i.e., $x = [\gb \mul h] = [\gb]$.
		Multiplication on the left by an element $g\in\symgrp$ sends fibres to fibres, and descends to $\Man$ to the action of $g$ to $\Man$, i.e., $[g \mul \gb] = g \act [\gb]$.
	}
\label{fig:homogeneousspace}
\end{figure}

We use the standard notations pertaining to calculus on smooth manifolds~\cite{AbMaRa1988}.
Given a manifold $\Man$, we denote its tangent bundle by $\Tan\Man$.
The space of sections of that tangent bundle (the vector fields) is denoted by $\diff$.
Given a function $f \colon \Man \to \mathcal{N}$, where $\Man$ and $\mathcal{N}$ are two smooth manifolds, we denote its push-forward, or Jacobian, by $\Tan f$, so $\Tan f \colon \Tan\Man \to \Tan \mathcal{N}$.

A \demph{homogeneous space} is a smooth manifold $\Man$, equipped with a transitive action of a Lie group $\symgrp$~\cite{KoNo69,Ki08,Sh97,Ga14}.
We denote this transitive action as
\begin{align}
g \act x
,
\qquad
g \in \symgrp
\quad
x\in\Man
.
\end{align}
Fixing an arbitray point $\origin\in\Man$, we define the \demph{isotropy group} at this point by
\begin{align}
	\isogrp \coloneqq \setc{g\in\symgrp}{g\act \origin = \origin}
	.
\end{align}
The manifold $\Man$ may be identified with cosets
\begin{align}
x \equiv g\mul \isogrp = \setc{g\mul h}{h \in \isogrp}
.
\end{align}
We denote the projection from $\symgrp$ to $\Man$ by
\begin{align}
[g] \coloneqq g\mul \isogrp
.
\end{align}
For a tangent vector $X$ at a point $g\in\symgrp$, we denote the projection by an abuse of notation
\begin{align}
[X] \coloneqq \Tan \pi \mul X
,
\end{align}
where $\pi$ denotes the projection $\pi(g) = [g]$.

Note that in order to simplify the notations, we will always work as if $\symgrp$ was a subgroup of $\GL$ for some integer $d$;
this is not a limitation because one can translate these notations into abstract manifold notations.
In particular, we will consider the \demph{Lie algebra} $\symalg$, the tangent space at the identity of $\symgrp$, to be an affine subspace of $\RR^{d\times d}$, the space of $d\times d$ matrices.
This allows us to define operations such as the adjoint action of $g \in \symgrp$ on $\xi \in \symalg$ by matrix multiplication as $\xi \to g \mul \xi \mul g\inv$.
Similarly, the multiplication $g \mul \xi$ denotes the tangent of the left multiplication map applied to $\xi\in\symalg$, and is thus considered as a vector at $g$, i.e., $g\mul \xi \in \Tan[g]\symgrp$.

Finally, we will repeatedly use that the space $\Tan[\origin]\Man$ is isomorphic to $\symalg/\isoalg$, which is straightforward to check~\cite[\S\,4.5]{Sh97}.

\section{Isotropy Maps}
\label{sec:isotropymaps}

An isotropy map, or choice of isotropy, essentially transforms a vector field $f$ into a set of frozen vector fields at every point on $\Man$.
We thus obtain the following definition.

\begin{definition}
We call an \demph{isotropy map} a linear map 
\begin{align}
	\isoc \in \Lin \paren[\big]{\diff , \Cinf(\Man, \symalg)}
\end{align}
which satisfies the \demph{consistency condition}
\begin{align}
\label{eq:isoconsistency}
	{\pairing{\isoc}{f}(x)}\act x = f(x) \qquad x \in \Man \quad f \in \diff
	.
\end{align}
\end{definition}

The interpretation is that, given a vector field $f$, the corresponding value $\nu = \pairing{\isoc}{f} \in \Cinf(\Man, \symalg)$ associates to every point $x\in\Man$ a \emph{frozen vector field} $F^x$ based at $x$ defined by the infinitesimal action
\begin{align}
	F^x(y) = \nu(x) \act y \qquad y \in \Man
	.
\end{align}

The {consistency condition}~\eqref{eq:isoconsistency} means that the frozen vector field for $f$ based at $x$, and evaluated at $x$, coincides with $f$ at $x$, i.e., $F^x(x) = f(x)$.

\subsection{Equivariance}

As $\symgrp$ naturally acts on $\Man$, it also acts on $\diff[\Man]$.
$\symgrp$ also acts on $\symalg$ by the adjoint action, and on $\Man$ by definition, so it acts on functions from $\Man$ to $\symalg$.
It is thus natural to ask for the equivariance of the isotropy map, which is defined as follows.

\begin{definition}
	An isotropy map is \demph{equivariant} if
	\begin{align}
		\pairing{\isoc}{g\act f} = g \act \pairing{\isoc}{f}
	\end{align}
	where the action of $g\in\symgrp$ on a vector field $f$ is  given by~\cite[\S\,2.4]{Ki08}
	\begin{align}
\label{eq:defactvf}
	(g\act f)(x) \coloneqq \Tan g \mul f(g\inv \act x)
	,
	\qquad
	x \in \Man
	,
	\end{align}
	the action on $\nu\in\Cinf(\Man, \symalg)$ is
	\begin{align}
\label{eq:defactfMG}
	(g \act \nu)(x) \coloneqq g\act \nu(g\inv \act x)
	,
	\qquad
	x\in\Man
	,
	\end{align}
	and the action on $\xi\in\symalg$ is the \emph{adjoint action}~\cite[\S\,2.6]{Ki08}
	\begin{align}
\label{eq:adjointaction}
	g \act \xi = g \mul \xi \mul g\inv
	.
	\end{align}
\end{definition}

Putting all the bits of the definitions together, equivariance of the isotropy map $\isoc$ is written as
\begin{align}
\pairing{\isoc}{g \act f} (x)
= g \mul \pairing{\isoc}{f}(g\inv\act x) \mul g\inv
,
\qquad
x\in \Man
,
\quad
g \in \symgrp
.
\end{align}

\subsection{Locality}

The locality of an isotropy map is the idea that the value of $\pairing{\isoc}{f}$ at $x$ may only depend on $f$ in an infinitely small neighbourhood of $x$.
Local isotropy maps are studied in~\cite{Ve14}.
To obtain a rigorous, tractable definition, we have to first introduce the support of a section.
See also~\cite{MKVe13,MLMoMKVe14} for similar definitions.

Recall that the \emph{support} $\supp(\nu)$ of a section $\nu$ of a vector bundle is the closure of the set where the section is non-zero.
For an isotropy choice $\nu\in\Cinf(\Man, \symalg)$, this gives
\begin{align}
	\supp(\nu) \coloneqq \overline{\setc{x \in \Man}{\nu(x) \neq 0}}
	,
\end{align}
and for a vector field $f\in\diff$,
\begin{align}
	\supp(f) \coloneqq \overline{\setc{x \in \Man}{f(x) \neq 0}}
	.
\end{align}

\begin{definition}
	An isotropy map
	\(
	\isoc \in \Lin \paren[\big]{\diff , \Cinf(\Man, \symalg)}
	\)
	is \demph{local} if it is support non increasing, i.e.,
	\begin{align}
		\supp\paren[\big]{\pairing{\isoc}{f}} \subset \supp(f)
		\qquad
		f \in \diff
		.
	\end{align}
\end{definition}

\subsection{Fundamental Example: the tautological isotropy map}
\label{sec:tautological}

A fundamental example, on any manifold $\Man$, is given by the action of the \emph{whole group of diffeomorphisms} $\symgrp = \Diff$.
The action is transitive.
Note that the Lie algebra $\symalg$ is now 
\(
	\symalg = \diff
\)%
,
where $\diff$ is the space of vector fields on $\Man$.

We define the \demph{tautological isotropy map}
\footnote{
	With the notation~\eqref{eq:notationconstiso} that isotropy choice is simply
	\(
		\pairing{\isoc}{f} = f
	\)
	hence the name of \emph{tautological} isotropy map.
}
by
\begin{align}
	\label{eq:constantisomap}
	\pairing{\isoc}{f}(x) \coloneqq f 
	,
	\qquad
	\forall x \in\Man
	,
	\quad
	f \in \diff
	.
\end{align}

The tautological isotropy map is an  \emph{equivariant} but \emph{non local} isotropy map.

Indeed, $\pairing{\isoc}{f}$ is a \emph{constant function} from $\Man$ to $\symalg = \diff$.
The support of a constant function is either $\Man$, or, if the constant is zero, the empty set.
This shows that $\isoc$ is \emph{not local}, as the support of a non-zero vector field need not be the whole manifold $\Man$.
The simple intuition behind the non-locality of the tautological isotropy map $\isoc$ is that the value at a point $x$ is the whole vector field $f$ defined everywhere, as opposed to some quantity calculated at $x$.

The reason this isotropy map is fundamental is that, as we shall see in \autoref{sec:exactsol}, it is the isotropy map which, assuming zero-order of the underlying skeleton, gives the exact solution.

\section{Runge--Kutta Methods on Homogeneous Spaces}
\label{sec:rungekutta}

We  put the known Lie group and homogeneous space methods in a unified framework.
This allows us to show in \autoref{prop:skelequi} that if isotropy map is equivariant, then so is the corresponding integrator.

\subsection{Skeletons}

\begin{figure}
	\centering

\begin{tikzpicture}[%
		every node/.style={draw=black,thick,anchor=west},
		levela/.style={fill=Gainsboro},
		levelb/.style={fill=MintCream},
		levelc/.style={fill=OldLace},
		leveld/.style={fill=LemonChiffon},
  grow via three points={one child at (0.5,-0.7) and
  two children at (0.5,-0.7) and (0.5,-1.4)},
  edge from parent path={(\tikzparentnode.south) |- (\tikzchildnode.west)}]
  \node[levela] {Method $\Meth = \skel\circ\isoc$}
    child { node[levelb] {Isotropy map $\isoc$}}		
    child { node[levelb]  {Skeleton $\skel$}
      child { node[levelc] {Stage Tree $\Gamma$}}
      child { node[levelc] {Motions $\motion_{i,j}$}
	  	child {node[leveld] {Motion map $\motion$}}
		child {node[leveld] {Transition functions $\trans_{i,j}$}}
			}
		}
	;
\end{tikzpicture}
\caption{
	A summary of the relation between the constitutive elements of a method on a homogeneous space.
}
\label{fig:skeleton}
\end{figure}
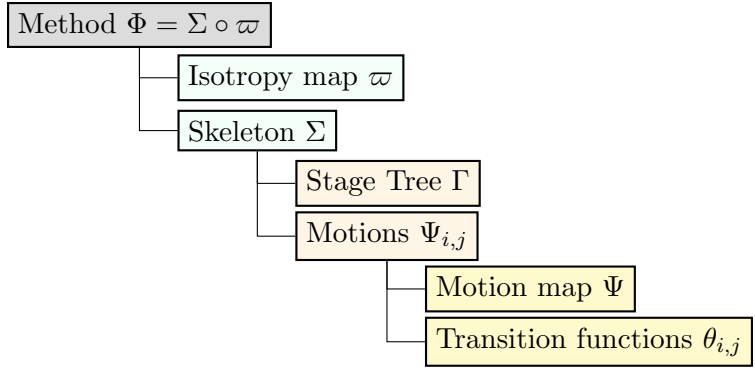

We define a Runge--Kutta skeleton for a given Lie group $\symgrp$ from the following ingredients.
\begin{description}
	\item[Stage Tree $\Gamma$] Tree (i.e., a connected, undirected graph with no cycles), in which two vertices are singled out: the \demph{initial vertex}, denoted ``$\initial$'' and the \demph{final vertex} denoted ``$\final$''.
	\item[Motions \(\motion_{i,j}\)]  Maps \begin{align}\motion_{i,j}\colon \symalg^n \to \symgrp\end{align} defined for any two adjacent vertices $i$, $j$ in the graph, and the compatibility \begin{align}\motion_{i,j} = \motion_{j,i}\inv
			.
		\end{align}
\end{description}

One defines a skeleton
\begin{align}
	\skel \colon \Cinf(\Man,\symalg) \to \Diff
\end{align}
using the \emph{stages} $X_i \in \Man$, and the \emph{frozen vector fields} $F_i \in \symalg$ as intermediate variables.

From a skeleton and an isotropy map, we obtain an integrator on any homogeneous manifold $\Man$ for which $\symgrp$ is the symmetry group.
Indeed, given an isotropy choice $\nu \in \Cinf(\Man, \symalg)$, the map
\begin{align}
	\skel(\nu) \in \Diff
\end{align}
is  defined by
\begin{align}
x_1 = \skel(\nu)(x_0)
\end{align}
as follows.

\begin{mdframed}[leftmargin=2em,rightmargin=2em,skipabove=2em, skipbelow=2em,
	frametitle={Runge--Kutta Method}]
	\begin{subequations}
		\label{eqs:skeleton}
\begin{align}
X_{\initial} &= x_0\\
X_i &= \motion_{i,j}(F) \act X_j & \text{$\forall$ edge $i,j$}\\
F_i &= \nu(X_i) & \text{$\forall$ vertex $i$}\\
x_1 &= X_{\final}
\end{align}
\end{subequations}
\end{mdframed}

In the sequel we will always assume that $\nu$ is scaled (with a time step), so that the equations \eqref{eqs:skeleton} implicitly defining $x_{1}$ have exactly one solution.

Let us consider the simplest possible case, a stage tree containing only the initial vertex $\initial$ and final vertex $\final$:
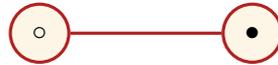
\begin{align}
	\begin{tikzpicture}[stagetree, ineq]
		\node[vertex, initfi] (init) {$\initial$};
		\node[vertex, initfi, right of=init] (final) {$\final$};
		\draw[norient, initfi] (init) -- (final);
	\end{tikzpicture}
\end{align}
The corresponding equations for an isotropy choice $\nu = \pairing{\isoc}{f}$ corresponding to an isotropy map $\isoc$, are
\begin{subequations}
\begin{align}
X_{\final} &= \motion_{\final,\initial}(F_{\initial}, F_{\final}) \act X_{\initial} ,\\
F_{\initial} &=  \pairing{\isoc}{f}\paren{X_{\initial}},\\
F_{\final} &= \pairing{\isoc}{f}\paren{X_{\final}},
\end{align}
\end{subequations}
and the integrator maps the initial condition $X_{\initial}$ to the point $X_{\final}$.

Possible values for $\motion_{\final,\initial}(F_{\initial}, F_{\final})$ are $\exp(F_{\initial})$, which gives the forward Euler method, $\exp(F_{\final})$, which gives the backward Euler method, or $\exp\paren[\big]{(F_{\initial}+F_{\final})/2}$ which gives the trapezoidal rule.

\subsection{Transition Functions}

Note that most often, the motions $\motion_{i,j}$ are defined by
\begin{align}
	\motion_{i,j} \coloneqq \motion \circ \trans_{i,j}
\end{align}
with
\begin{description}
	\item[Motion map \(\motion\)] \begin{align}\motion\colon\symalg\to\symgrp,\end{align} which is usually either the exponential map, or an approximation of it, or the Cayley map when $\symgrp$ is quadratic; it should have the property that \begin{align}\motion(-\xi) = \motion(\xi)\inv
			,
		\end{align}
		which holds both for the exponential and Cayley map.
\item[Transition functions \(\trans_{i,j}\)] For any edge $i,j$, the transition function \begin{align}\trans_{i,j}\colon \symalg^n \to \symalg,\end{align} with compatibility 
		\begin{align}
			\label{eq:transcompat}
			\trans_{i,j}(F) = - \trans_{j,i}(F)
			.
		\end{align}
\end{description}

We will give the full expression of $\trans_{i,j}$ on the stage tree as
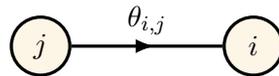
\begin{align}
	\begin{tikzpicture}[stagetree, ineq]
		\node[vertex] (init) {$j$};
		\node[vertex, right of=init] (final) {$i$};
		\draw[orient] (init) -- (final) node[midway, auto]{$\trans_{i,j}$};
	\end{tikzpicture}
\end{align}

For example, the \emph{forward Euler method} is given by
\begin{align}
	\label{eq:fEuler}
	x_1 = \exp\paren[\big]{\pairing{\isoc}{f}\paren{x_0}} \act x_0
\end{align}
so it corresponds to the transition function
\begin{align}
	\label{eq:transeeuler}
	\trans_{\final,\initial} = F_{\initial}
\end{align}
and it is equivalent to the tree
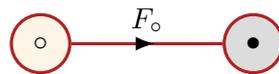
\begin{align}
	\begin{tikzpicture}[stagetree, ineq]
		\node[vertex, initfi] (init) {$\initial$};
		\node[vertex, initfi, intermediate, right of=init] (final) {$\final$};
		\draw[orient, initfi] (init) -- (final) node[midway, auto]{$F_{\initial}$};
	\end{tikzpicture}
\end{align}

Note that \eqref{eq:transcompat} means that the orientation is arbitrary, so with a different orientation, we obtain exactly \emph{the same method}, as $  \trans_{\initial,\final} = -\trans_{\final, \initial}$:
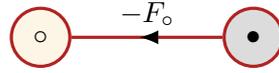
\begin{align}
	\begin{tikzpicture}[stagetree, ineq]
		\node[vertex, initfi, ] (init) {$\initial$};
		\node[vertex, initfi, intermediate, right of=init] (final) {$\final$};
		\draw[orient, initfi] (final) -- (init) node[midway, above]{$-F_{\initial}$};
	\end{tikzpicture}
\end{align}

We shade the vertices which value is not used in any of the transition functions.
In the case above, the value $F_{\final}$ is never needed, so the vertex $\final$ is shaded.

With these notations, the \emph{backward Euler method} is given by the transition function $\trans_{\final,\initial} = F_{\final}$, so we write
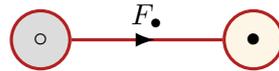
\begin{align}
	\begin{tikzpicture}[stagetree, ineq]
		\node[vertex, initfi, intermediate] (init) {$\initial$};
		\node[vertex, initfi, right of=init] (final) {$\final$};
		\draw[orient, initfi] (init) -- (final) node[midway, auto]{$F_{\final}$};
	\end{tikzpicture}
\end{align}

For the \emph{trapezoidal rule}, the transition function is $\trans_{\final,\initial} = (F_{\initial}+F_{\final})/2$, so it is written as:
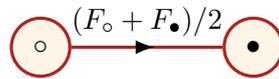
\begin{align}
	\begin{tikzpicture}[stagetree]
		\node[vertex, initfi] (init) {$\initial$};
		\node[vertex, initfi, right of=init] {$\final$};
		\draw[orient, initfi] (init) -- (final) node[midway, auto]{$(F_{\initial} + F_{\final})/2$};
	\end{tikzpicture}
\end{align}

A slightly more involved example is the \emph{implicit midpoint rule}, which requires an extra stage which has the arbitrary label ``$\star$''.
The stage tree has thus three vertices, $\initial$, $\final$ and $\star$.
The rules for that method are to use one half of the frozen vector field at $\star$, and use it to go from $\initial$ to $\star$, as well as from $\star$ to $\final$.
The transition functions are thus
\begin{subequations}
\begin{align}
	\trans_{\star,\initial} &= F_{\star}/2\\
	\trans_{\final,\star} &= F_{\star}/2
\end{align}
\end{subequations}
so we write
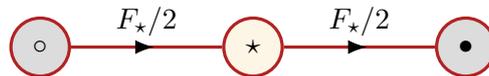
\begin{align}
	\begin{tikzpicture}[stagetree]
		\node[vertex, initfi] (star) {$\star$};
		\node[vertex, initfi, intermediate, left of=star] (init) {$\initial$};
		\node[vertex, initfi, intermediate, right of=star] (final) {$\final$};
		\draw[orient, initfi] (init) -- (star) node[midway, above]{$F_{\star}/2$};
		\draw[orient, initfi] (star) -- (final) node[midway, auto]{$F_{\star}/2$};
	\end{tikzpicture}
\end{align}
Note again that the values $F_{\initial}$ and $F_{\final}$ are not used, so we shade the corresponding vertices $\initial$ and $\final$.
We also emphasise the (unique) path between the initial and final vertices.

For the sake of completeness, we also give the corresponding equations, which are
\begin{subequations}
\begin{align}
X_{\star} &= \exp(F_{\star}/2)\act X_{\initial} \\
X_{\final} &= \exp(F_{\star}/2)\act X_{\star} 
\end{align}
\end{subequations}
and
\begin{align}
F_{\star} &= \pairing{\isoc}{f}(X_{\star}) 
\end{align}
According to our convention, the integrator maps $X_{\initial}$ to $X_{\final}$.
Note that $X_{\initial}$, $X_{\final}$ and $X_{\star}$ are points in $\Man$, and that $F_{\star}\in\symalg$.

\subsection{Examples}
\label{sec:skeletonexamples}

We give examples of skeletons for some standard methods such as the standard Runge--Kutta methods, RKMK, Crouch--Grossmann and commutator free.

\subsubsection{Detailed example}

We start with a detailed example,
 the fourth order commutator-free method \cite{CeMaOw03}.
It is defined by the following transition functions.

\begin{subequations}
\begin{align}
	\trans_{1,\initial} &= \frac{1}{2}F_{\initial}
	\\
	\trans_{2,\initial} &= \frac{1}{2}F_1
	\\
	\trans_{3,1} &= -\frac{1}{2} F_{\initial} + F_2
	\\
	\trans_{4',\initial} &= \frac{1}{12}(3F_{\initial} + 2(F_1+F_2) - F_3) 
	\\
	\trans_{\final,4'} &= \frac{1}{12}(-F_{\initial} + 2(F_1+F_2) + 3F_3) 
\end{align}
\end{subequations}

It means that we have the vertices $\initial$, $1$, $2$, $3$, $4'$, and $\final$.
Each equation gives an edge, so there is an edge between $\initial$ and $1$, between $\initial$ and $2$, etc.
The stage tree thus takes the following form
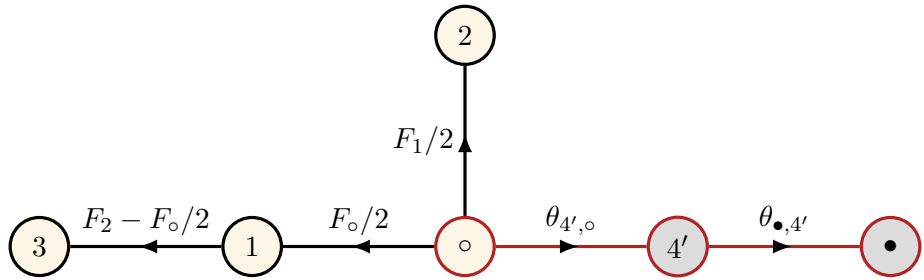
\begin{align}
	\begin{tikzpicture}[
		stagetree,
		]
		\node[vertex] (3) {$3$};
		\node[vertex] (1) [right of=3] {$1$};
		\node[vertex, initfi] (initial) [right of=1] {$\initial$};
		\node[vertex, intermediate, initfi] (4p) [right of=initial] {$4'$};
		\node[vertex, initfi, intermediate] (final) [right of=4p] {$\final$};
		\node[vertex] (2) [above of=initial] {$2$};
		\draw[orient] (initial) -- (1) node[midway, above]{$F_{\initial}/2$};
		\draw[orient] (1) -- (3) node[midway, above]{$F_2-F_{\initial}/2$};
		\draw[orient] (initial) -- (2) node[midway, auto]{$F_1/2$};
		\draw[orient, initfi] (initial) -- (4p) node[midway,auto]{$\trans_{4',\initial}$};
		\draw[orient, initfi] (4p) -- (final) node[midway,auto]{$\trans_{\final, 4'}$};
	\end{tikzpicture}
\end{align}
As mentioned earlier, we also indicate the intermediate vertices, i.e., the vertices for which the component in the $F$ vector is not computed.
For instance in the current example, the values $F_{4'}$ and $F_{\final}$ are never needed.
Note that the vertex numbering is in general arbitrary, but in this case it indicates the order of the equation for which the method is explicit.

For completeness, we give the complete set of equation for the final method.
Recall that each edge gives rise to an equation following the rule \eqref{eqs:skeleton}.
First one has to choose an isotropy map $\nu$, which can for instance be defined as in \eqref{eq:defisoconn} from one of the connections in \autoref{tab:homogconn}.
In order to obtain a point $x_1$ from an initial condition $x_0$, we have to solve the following equations, with $\nu = \pairing{\isoc}{f}$.
\begin{subequations}
	\begin{align}
	x_0 &= X_{\initial} & F_{\initial} = \nu(X_{\initial}) \\
	X_{1} &= \exp\paren[\big]{F_{\initial}/2} \act X_{\initial} & F_1 = \nu(X_1) \\
	X_{2} &= \exp\paren[\big]{F_1/2} \act X_{\initial} & F_2 = \nu(X_2) \\
	X_{3} &= \exp\paren[\big]{F_2 - F_{\initial}/2} \act X_{1} & F_3 = \nu(X_3) \\
	X_{4'} &= \exp\paren[\Big]{\frac{1}{4}F_3 + \frac{1}{6}(F_1 + F_2) - \frac{1}{12}F_{\initial} } \act X_{\initial} \\
	X_{\final} &= \exp\paren[\Big]{\frac{1}{4}F_{\initial} + \frac{1}{6}(F_1+F_2) - \frac{1}{12}F_3 } \act X_{4'} \\
	x_1 &= X_{\final}
	\end{align}
\end{subequations}

\subsubsection{Runge--Kutta methods in $\RR^d$}
A regular Runge--Kutta method on $\RR^d$ with Butcher tableau $A$, $b$ \cite[\S\,II.1.1]{HaLuWa06} has the following skeleton
\begin{itemize}
	\item The motion map is $\Psi = \exp$
	\item The transition functions are
            \begin{subequations}
		\begin{align}
			\trans_{i,\initial} &\coloneqq \sum_{j=1}^{s}a_{i,j} F_j \qquad 1 \leq i \leq s\\
			\trans_{\final,\initial} &\coloneqq \sum_{j=1}^{s} b_j F_j
		\end{align}
            \end{subequations}
\end{itemize}

\subsubsection{RKMK3}
A third order RKMK method can be encoded as follows:
\begin{align}
	\trans_{1,\initial} &= \frac{1}{2}F_{\initial}\\
	\trans_{2,\initial} &= -F_0 + 2F_1\\
	\trans_{\final,\initial} &= \frac{1}{6}(F_0 + 4F_1+F_2) + \frac{1}{36}[4F_1+F_2,F_{\initial}] \\
\end{align}

The corresponding stage tree is thus:
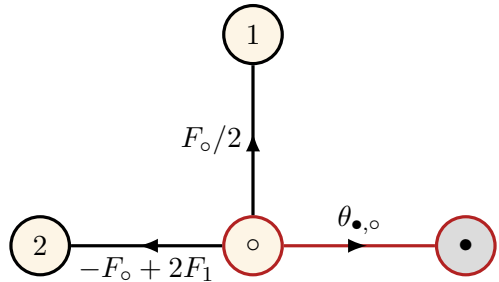
\begin{equation}
\begin{tikzpicture}[stagetree, ineq]
	\node[vertex,  initfi] (init) {$\initial$};
	\node[vertex, left of=init] (2) {$2$};
	\node[vertex, initfi, intermediate, right of=init] (final) {$\final$};
	\node[vertex, above of=init] (1) {$1$};
	\draw[orient, initfi] (init) -- (final) node[midway, above]{$\trans_{\final,\initial}$};
	\draw[orient] (init) -- (1) node[midway, auto]{$F_{\initial}/2$};
	\draw[orient] (init) -- (2) node[midway, auto] {$-F_{\initial} + 2 F_{1}$};
\end{tikzpicture}
\end{equation}

\subsubsection{RKMK4}
A fourth order RKMK method can be encoded as follows:
\begin{subequations}
\begin{align}
	\trans_{1,\initial} &= \frac{1}{2}F_{\initial}\\
	\trans_{2,\initial} &= \frac{1}{2}F_1 - \frac{1}{8}\bracket{F_{\initial},F_1}\\
	\trans_{3,\initial} &= F_2 \\
	\trans_{\final,\initial} &= \frac{1}{6}\paren[\big]{F_{\initial} + 2(F_1+F_2) +F_3} - \frac{1}{12} \bracket{F_0, F_3}
\end{align}
\end{subequations}

We give the corresponding stage tree for completeness:

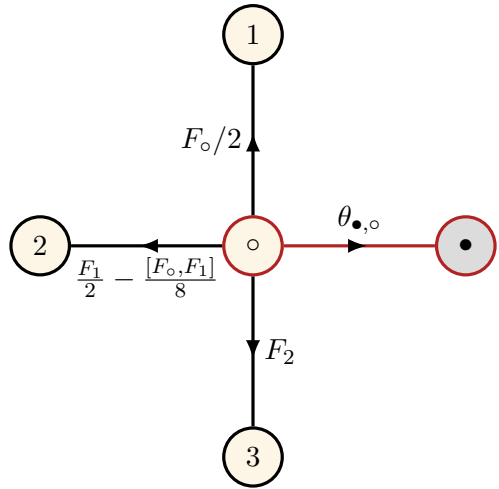
\begin{equation}
\begin{tikzpicture}[stagetree, ineq]
	\node[vertex,  initfi] (init) {$\initial$};
	\node[vertex, left of=init] (2) {$2$};
	\node[vertex, initfi, intermediate, right of=init] (final) {$\final$};
	\node[vertex, above of=init] (1) {$1$};
	\node[vertex, below of=init] (3) {$3$};
	\draw[orient, initfi] (init) -- (final) node[midway, above]{$\trans_{\final,\initial}$};
	\draw[orient] (init) -- (1) node[midway, auto]{$F_{\initial}/2$};
	\draw[orient] (init) -- (2) node[midway, auto] {$\frac{F_1}{2} - \frac{[F_{\initial},F_1]}{8}$};
	\draw[orient] (init) -- (3) node[midway, auto] {$F_2$};
\end{tikzpicture}
\end{equation}

\subsubsection{CG3}
The third order Crouch--Grossman method \cite[\S,IV.8.1]{HaLuWa06} is given by the following transition functions:
\begin{subequations}
\begin{align}
	\trans_{1,\initial} = \frac{3}{4}F_{\initial}\\
	\trans_{2,2'} = \frac{17}{208}F_1 \qquad
	\trans_{2',\initial} = \frac{119}{216}F_{\initial} \\
	\trans_{\final,3'} = \frac{24}{17}F_2 \qquad
	\trans_{3',3''} = -\frac{2}{3}F_1\qquad
	\trans_{3'',\initial} = \frac{13}{51}F_{\initial}
\end{align}
\end{subequations}

\subsubsection{Symmetric Gauss of order four}
The symmetric Gauss method \cite{ZaEnMK01} can also be encoded in this way.
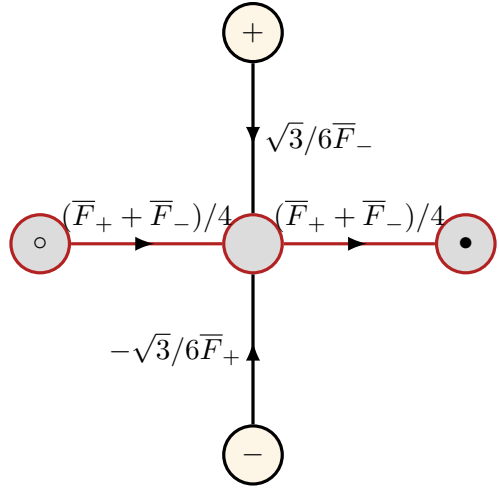
\newcommand*\bF{\overline{F}}

\begin{align}
\begin{tikzpicture}[stagetree, ineq]
	\node[vertex, intermediate, initfi] (intermediate) {};
	\node[vertex, initfi, intermediate, left of=intermediate] (init) {$\initial$};
	\node[vertex, initfi, intermediate, right of=intermediate] (final) {$\final$};
	\node[vertex, above of=intermediate] (plus) {$+$};
	\node[vertex, below of=intermediate] (minus) {$-$};
	\draw[orient, initfi] (init) -- (intermediate) node[midway, above]{$(\bF_+ + \bF_-)/4$};
	\draw[orient, initfi] (intermediate) -- (final) node[midway, above]{$(\bF_+ + \bF_-)/4$};
	\draw[orient] (minus) -- (intermediate) node[midway, auto] {$-\sqrt{3}/6 \bF_+$};
	\draw[orient] (plus) -- (intermediate) node[midway, auto] {$\sqrt{3}/6 \bF_-$};
\end{tikzpicture}
\end{align}

The auxiliary variables $\bF_{+}$ and $\bF_{-}$ are determined implicitly by the equation
\begin{subequations}
\begin{align}
	\bF_{+} &= F_{+} +\frac{\sqrt{3}}{4} \bracket{\bF_{-},F_{+}}\\
	\bF_{-} &= F_{-} - \frac{\sqrt{3}}{4} \bracket{\bF_{+}, F_{-}}
\end{align}
\end{subequations}

\subsection{Order Zero and Exact Solution}
\label{sec:exactsol}

For a given $\xi \in \symalg$, we define the \demph{constant isotropy choice} $\nu\in\Cinf(\Man,\symalg)$ as \begin{align}
	\nu \colon x \mapsto \xi
	,
	\qquad
	x \in \Man
	.
\end{align}
We use the abuse of notation 
\begin{align}
	\label{eq:notationconstiso}
\nu = \xi
\end{align}
 in this case.

\begin{definition}
	We say that a skeleton $\skel$ has \demph{order zero} if it gives the exact solution for constant isotropy choices, i.e., using the notation \eqref{eq:notationconstiso},
	\begin{align}
		\skel(\xi) = \exp(\xi) \qquad \xi \in \symalg
		.
	\end{align}
\end{definition}

All the skeletons in \autoref{sec:skeletonexamples} have  order zero, so one obtains the exact solution with the motion map $\exp$,  and using the tautological isotropy map \eqref{eq:constantisomap} in any of those skeletons.

\subsection{Equivariance of Skeletons}

The following result shows that in practice it suffices to check the equivariance of the motion map $\motion$ and the transition functions $\trans_{i,j}$ to obtain an equivariant skeleton, and thus an equivariant method when used with an equivariant isotropy map.

\begin{proposition}
	\label{prop:skelequi}
	We have the following ``trickling down'' results on equivariance:
	\begin{enumerate}[label=\upshape(\roman*)]
		\item if $\motion$ and $\trans_{i,j}$ are equivariant, then so are $\motion_{i,j} = \motion \circ \trans_{i,j}$
		\item 
			\label{it:equiskel}
			if $\motion_{i,j}$ are equivariant, then so is the skeleton $\skel$
		\item if the skeleton $\skel$ and the isotropy map $\isoc$ are equivariant, then so is the method $\Meth = \skel \circ \isoc$.
	\end{enumerate}
\end{proposition}
\begin{proof}
	We give a proof of \autoref{it:equiskel}, the other two statement being clear.
	We simply show that if $x_1 = \skel(\nu)(x_0)$ is a solution with the stages $X_i$ and frozen vector fields $F_i$, then we have $g\act x_1 = \skel(g \act \nu)(g\act x_0)$.
	\begin{subequations}
\begin{align}
g \act X_i &= g\act(\motion_{i,j}(F) \act X_j) \\
&= g \act(\motion_{i,j}(F) \mul g\inv \mul g\act X_j)\\
&= \motion_{i,j}(g\act F)\act g\act X_j 
\end{align}
\end{subequations}
\end{proof}

Note that in all practical examples, the motion maps are of the form $\motion_{i,j} = \motion \circ \trans_{i,j}$.
The motion map $\motion$ is the exponential, or Cayley map, both of which are equivariant.
The transition functions $\trans_{i,j}$ are Lie-algebra morphisms, as they are linear combinations of commutators.

This gives the following result:
\begin{proposition}
	\label{prop:RKMKCGCFequi}
	All the RKMK, Crouch--Grossman and commutator-free skeletons are equivariant.
\end{proposition}

\subsection{Locality}

We show the relation between locality of the skeleton and of the isotropy map.
First, we give a practical way of checking that a skeleton is local.

\begin{definition}
	We say that the motion map $\motion$ is local if
	\begin{align}
		\motion(0) = \Id
	\end{align}
	We say that the motion maps $\motion_{i,j}$ are local if
	\begin{align}
		\motion_{i,j}(0) = \Id
	\end{align}
	We say that the transition functions $\trans_{i,j}$ are local if
	\begin{align}
		\trans_{i,j}(0) = 0
	\end{align}
\end{definition}

\begin{proposition}
	\label{prop:skellocal}
	We have the following ``trickle down'' result on locality.
	\begin{enumerate}[label=\upshape(\roman*)]
		\item if the motion map $\motion$ and the transition functions $\trans_{i,j}$ are local, then so are the movement maps $\motion_{i,j} = \motion \circ \trans_{i,j}$
		\item if the motion maps $\motion_{i,j}$ are local, then so is the correspondent skeleton $\skel$.
		\item if the skeleton $\skel$ and the isotropy map $\isoc$ are local, then so is the method $\Meth = \skel \circ \isoc$.
	\end{enumerate}
\end{proposition}
\begin{proof}
	The only non trivial statement is the last one.
	Suppose that $f(x) = 0$ in a neighbourhood of the initial condition $x_0$.
        By definition of locality, we then have $\pairing{\isoc}{f}(x_0) = 0$.
	It is then easy to check that taking all the stages $X_i$ equal to $x_0$ provides a solution, and in particular, we obtain $x_1 = x_0$, which finishes the proof.
\end{proof}

As we already noticed in \autoref{sec:tautological}, the tautological isotropy map is \emph{not} local.
We saw however in \autoref{sec:exactsol} that the corresponding ``method'' is the exact solution
As the exact solution is a local method, the locality of the isotropy map $\isoc$ is \emph{not} necessary.

\subsection{Relation between Lie group integrators and homogeneous space integrators}

An isotropy map $\isoc$ allows us to \emph{lift} a vector field $f\in\diff$  
to a {lifted vector field} in $\diff[\symgrp]$ in two ways:
\begin{align}
	\label{eq:deflifted}
	\lifted(g) \coloneqq \pairing{\isoc}{f}\paren{[g]} \mul g \qquad \lifted[-][\isoc](g) \coloneqq - g \mul \pairing{\isoc}{f}([g])
	.
\end{align}

One can use the lifting property of the isotropy map to obtain an integrator in the group $\symgrp$ instead.
We now show that this integrator \emph{descends} to the integrator of the same type on $\Man$.

\begin{proposition}
	\label{prop:lifting}
	Assume an equivariant skeleton $\skel$ is defined over a group $\symgrp$.
	Then the integrator $\skel(\conn_{\pm})(\lifted[\pm])$ (where $\conn_{\pm}$ is one of the Maurer--Cartan forms defined in \autoref{sec:liegroupact}) descends to the integrator $\skel(\isoc)(f)$ on the homogeneous manifold $\Man$, i.e., the following diagram commutes:
	\begin{align}
\begin{tikzpicture}[ineq]
\matrix(m) [commdiag, column sep=5em, row sep=4em]
{\symgrp \& \symgrp \\
	\Man \&  \Man \\};
\path[->]
(m-1-1) edge node[auto] {$\skel(\conn_{\pm})(\lifted[\pm])$} (m-1-2)
(m-1-1) edge (m-2-1)
(m-1-2) edge (m-2-2)
(m-2-1) edge node[auto] {$\skel(\isoc)(f)$} (m-2-2)
;
\end{tikzpicture}
	\end{align}
\end{proposition}
\begin{proof}
	On the edge $i,j$, we obtain the equation
	\(
	G_i = \motion_{i,j}({F}) \mul G_j
	\),
	so, as $[G_i] = G_i \mul \isogrp$, we have
	\(
		[G_i] = \motion_{i,j}(F) \act [G_j]
	\).
	Notice now that by definition of the lifted vector field $\lifted[\pm]$, we have $\pairing{\conn_{\pm}}{\lifted[\pm]}_{g} = \pairing{\isoc}{f}_{[g]}$, so
	\(
		{F}_i = \pairing{\conn_{\pm}}{\lifted[\pm]}_{G_i} = \pairing{\conn}{f}_{[G_i]}
	\),
	and we conclude that $[G_{\final}]$ is the image of $[G_{\initial}]$ by $\skel(\isoc)(f)$.
\end{proof}

\section{Zero-order Equivariant Isotropy Maps: Connections and Reductivity}
\label{sec:connections}

We study the isotropy maps that are one-forms, i.e., that only depend on the value of the vector field at a given point $x\in\Man$, and not on its derivatives.
These forms can thus be regarded as ``zero order'' isotropy maps.
The main results of this section are summed up in \autoref{fig:connections}.

\begin{figure}
	\centering
\begin{tikzpicture}[
	concept/.style={shape=circle, fill=OldLace, draw=black, text width=7em, align=center, thick},
	rotation/.style={>=stealth, ->, very thick, gray},
	birotation/.style={rotation, <->,},
	equivalence/.style={implies-implies, double, thick, double equal sign distance},
scale=2.5]
 
\begin{scope}[rotate=90, scale=1.3]
    \coordinate (Z) at (0:1);
    \coordinate  (P) at (120:1);
    \coordinate (M) at (-120:1);
\end{scope}

	\node[concept,] (redstruc) at (P) {Reductive\\structure};
	\node[concept,] (princconn) at (M) {Invariant\\principal\\connection};
	\node[concept,] (conn) at (Z) {Connection\\{\scriptsize Zero order equivariant isotropy map}};

	\draw[equivalence] (redstruc) to node[sloped,midway,above]{\autoref{sec:reductive}} (conn);
	\draw[equivalence] (conn) to node[sloped, midway, above]{\autoref{sec:princconn}} (princconn);
	\draw[equivalence] (princconn) to node[auto]{Nomizu \cite{No54}} (redstruc);

	\node[concept, fill=Lavender] at (0,0) {$\Lin_{\isogrp}(\bundle{\symalg/\isoalg}{\symalg})$};

 
 
\end{tikzpicture}
\caption[Concept Equivalence]{
	There is an affine bijection between the affine space of connections, of invariant connections and of reductive structures.
	All these spaces are isomorphic to the affine space $\Lin_{\isogrp}(\bundle{\symalg/\isoalg}{\symalg})$ of $\isogrp$-invariant linear sections from $\symalg/\isoalg$ to $\symalg$.
}
\label{fig:connections}
\end{figure}

\subsection{Connections}

We define connections as equivariant $\symalg$-valued one-forms fulfilling a consistency condition.

\begin{definition}
We define an \demph{isotropy form} as a $\symalg$-valued one-form
\begin{align}
	\conn \in \Omega^1(\Man, \symalg)
\end{align}
satisfying the \demph{consistency condition}
\begin{align}
	\label{eq:conncons}
	\pairing{\conn}{f}_x \act x = f(x)
	.
\end{align}
\end{definition}

It is immediate that the set of isotropy forms has an \emph{affine structure}, i.e., if $\conn_1$ and $\conn_2$ are isotropy forms, then so is $\theta \conn_1 + (1-\theta)\conn_2$ for any real $\theta$.

\begin{definition}
\label{def:connection}
We say that an isotropy form $\conn$ is a \demph{connection} if it is \emph{equivariant}:
\begin{align}
	\label{eq:connequi}
	\pairing{\conn}{\Tan g \mul f}_{g\act x} = g \act \pairing{\conn}{f}_{x}
	\qquad
	g \in \symgrp
	,
	\quad
	x\in \Man
	,
\end{align}
where  the action of $g$ on $\symalg$ is the adjoint action~\eqref{eq:adjointaction}.
\end{definition}

Note that our definition differs from that of an invariant principal connection, and the relation between the two notions is detailed in \autoref{sec:princconn}.

Connections are, in a sense, the simplest possible local equivariant isotropy maps, because they are of \emph{order zero}, i.e., they depend only on the value of a vector field at a point, and not on its higher order derivatives.
More general local, 
equivariant isotropy maps are considered in {\cite{Ve14}}.

\begin{proposition}
	Given an isotropy form $\conn\in\Omega^1(\Man,\symalg)$, define the isotropy map $\isoc$ by $\pairing{\isoc}{f} \coloneqq \imath_f \conn$, that is
	\begin{align}
		\label{eq:defisoconn}
		\pairing{\isoc}{f}(x) \coloneqq \pairing{\conn}{f}_x
		.
	\end{align}
	The map $\isoc$ is then a local isotropy map.
	The isotropy map $\isoc$ is equivariant if and only the isotropy form $\conn$ is equivariant, that is, if $\conn$ is a connection.
\end{proposition}
\begin{proof}
	If $f$ is zero in a neighbourhood of $x\in\Man$, then, in particular, $f(x) = 0$, so $\pairing{\isoc}{f}(x) = \pairing{\conn}{f}_x = 0$, which proves the locality property.

	We note, using the definitions~\eqref{eq:defisoconn},~\eqref{eq:defactvf}, \eqref{eq:connequi} and \eqref{eq:defactfMG}, that 
	\begin{align}
	\pairing{\isoc}{g\act f}(x) 
	= \pairing{\conn}{g\act f}_x
	= \pairing{\conn}{\Tan g \mul f}_{g\inv x}
	= g \act \pairing{\conn}{f}_{g\inv x}
	= \paren[\big]{g\act \pairing{\isoc}{f}}(x)
	,
	\end{align}
	which shows that $\isoc$ is equivariant if and only if $\conn$ is.
\end{proof}

Note that a connection defines at the origin a linear map
\begin{align}
	\label{eq:defconnorig}
	\conn_{\origin} \in \Lin\paren{\symalg/\isoalg , \symalg}
\end{align}
where we use the canonical identification 
\begin{align}
\symalg/\isoalg \equiv \Tan[\origin]\Man
.
\end{align}
The map $\conn_{\origin}$ is $\isogrp$-equivariant, so
\begin{align}
	\conn_{\origin} \in \Lin_{\isogrp}(\symalg/\isoalg, \symalg),
\end{align}
where, in general,
\begin{align}
	 \Lin_{\isogrp}(V, W) \coloneqq \setc{\varphi \in \Lin(V,W)}{\varphi(h\act v) = h\act\varphi(v)\qquad h\in\isogrp}
\end{align}
denotes the $\isogrp$-equivariant linear maps from a vector space $V$ to a vector space $W$, both equipped with a linear $\isogrp$-action.

Now, the infinitesimal action of $\symalg$ on $\Tan[\origin]\Man \equiv \symalg/\isoalg$ is just the projection $\xi \to \xi + \isoalg$.
We obtain that the consistency condition~\eqref{eq:conncons} becomes
\begin{align}
	\pairing{\conn_{\origin}}{f} + \isoalg = f
	,
	\qquad
	f \in \symalg/\isoalg
	,
\end{align}
so we interpret $\conn_{\origin}$ as a \emph{linear section} of the projection $\symalg \to \symalg/\isoalg$.
We denote the corresponding affine space
\begin{align}
	\Lin(\bundle{\symalg/\isoalg}{\symalg}) \coloneqq \setc[\big]{\omega \in \Lin(\symalg/\isoalg, \symalg)}{\pairing{\omega}{f} + \isoalg = f \quad \forall f \in \symalg/\isoalg}
	.
\end{align}

The affine space of $\isogrp$-invariant sections is denoted by
\begin{align}
	\Lin_{\isogrp}(\bundle{\symalg/\isoalg}{\symalg}) \coloneqq \Lin(\bundle{\symalg/\isoalg}{\symalg}) \cap \Lin_{\isogrp}(\symalg/\isoalg, \symalg)
	.
\end{align}

\begin{proposition}
\label{prop:connequi}
	The map from the affine space of connections to
	\begin{align}
		\Lin_{\isogrp}(\bundle{\symalg/\isoalg}{\symalg})
	\end{align}
	defined by $\conn \mapsto \conn_{\origin}$
	where $\conn_{\origin}$ is defined in~\eqref{eq:defconnorig},
	is an affine bijection.

\end{proposition}
	The proof is elementary, and is a simplified version of the extension principle presented in details in~\cite[\S\,4.2]{MKVe13}.
\begin{proof}
	Pick an element $\nu \in \Lin_{\isogrp}(\bundle{\symalg/\isoalg}{\symalg})$.
	It defines a connection defined at the point $x = g \act \origin$ by
	\(
			\pairing{\conn}{f}_{g \act \origin} = g \act \pairing{\nu}{\Tan g\inv f}
	\).
	One checks that the $\isogrp$-equivariance of $\nu$ ensures that this map is well defined, i.e., does not depend on which element $g\in\symgrp$ is chosen such that $x = g \act \origin$.
\end{proof}

\subsection{Reductive decompositions}
\label{sec:reductive}

We proceed to show the relation with the existing concept of reductive decompositions.

\begin{definition}
A \demph{reductive decomposition}~\cite{No54} is a decomposition
\begin{align}
	\symalg = \isoalg \oplus \redalg
\end{align}
such that
\begin{align}
	\label{eq:redequi}
H \act \redalg \subset \redalg
.
\end{align}
\end{definition}

Note that if $\isogrp$ is simply connected, the condition~\eqref{eq:redequi} is equivalent to
\begin{align}
	\bracket{\isoalg, \redalg} \subset \redalg
	.
\end{align}

We identify a complementary subspace to $\isoalg$ as a {section} of the projection $\symalg \to \symalg/\isoalg$, that is a linear map
\begin{align}
	\redsec \colon \symalg/\isoalg \to \symalg
\end{align}
such that
\begin{align}
	\redsec(x)+\isoalg = x \qquad x \in \symalg/\isoalg
\end{align}

We have the following relation between equivariance and reductivity:
\begin{lemma}
	The section $\redsec$ is reductive if and only if it is $\isogrp$-equivariant, that is
	\begin{align}
		h \act \redsec(x) = \redsec(h \act x)
	\end{align}
\end{lemma}
\begin{proof}
	As $\redsec$ is reductive, we have that
	\begin{align}
		h\act \redsec(x) = \redsec(x')
	\end{align}
	for some $x' \in \symalg/\isoalg$.
	But since $\redsec$ is a section, we have
	\begin{align}
		x' = \redsec(x') + \isoalg = h \act \redsec(x) + \isoalg = h \act (\redsec(x) + \isoalg) = h \act x
	\end{align}
	which proves the claim.
\end{proof}

We thus obtain that reductive structures are equivalent to connections.

\begin{proposition}
\label{prop:reducaff}
	The set of reductive structures is in affine bijection with
\begin{align}
	\Lin_{\isogrp}(\bundle{\symalg/\isoalg}{\symalg})
\end{align}
\end{proposition}

A consequence of \autoref{prop:connequi} and \autoref{prop:reducaff} is the following result.

\begin{proposition}
The affine space of connections and the set of reductive structures are in affine bijection with the affine space $\Lin_{\isogrp}(\bundle{\symalg/\isoalg}{\symalg})$.
	 The underlying linear space is
	\begin{align}
		\Lin_{\isogrp}(\symalg/\isoalg, \isoalg)
		.
	\end{align}
\end{proposition}

\subsection{Invariant Principal Connections and Horizontal Lifts}
\label{sec:princconn}

Recall that a homogeneous space is a particular instance of a \emph{principal $\isogrp$-bundle}~\cite{Sh97, St64}.
In that context, a \emph{principal connection} is a $\isoalg$-valued one-form $\theta$ which is $\isogrp$-equivariant in the sense that \begin{align}\label{eq:hequivariance}\pairing{\theta}{X \mul h}_{g \mul h} = h\inv \mul \pairing{\theta}{X}_{g} \mul h,\qquad h\in\isogrp.\end{align}
To be a principal connection, $\theta$ must also satisfy the consistency condition
\begin{align}
	\label{eq:princconnconscond}
	\pairing{\theta}{g \mul \xi}_{g} = \xi 
	,
	\qquad
	\xi\in\isoalg
	.
\end{align}

Finally, such a principal connection \(\theta\) is \emph{invariant} if
\begin{align}
	\pairing{\theta}{\gb \mul X}_{\gb\mul g}  = \pairing{\theta}{X}_{g}
	,
	\qquad \gb,g \in \symgrp
	.
\end{align}

\begin{proposition}
\label{prop:princconn}
Consider a connection $\conn$ as defined in \autoref{def:connection}.
The $\symalg$-valued one-form $\theta$ defined by
\begin{align}
	\pairing{\theta}{X}_{g} \coloneqq g\inv \mul X - g\inv \mul \pairing{\conn}{[X]}_{[g]} \mul g 
	,
	\qquad
	g \in \symgrp
	,
	\quad
	X \in \Tan[g]\symgrp
	.
\end{align}
takes values in $\isoalg$ and is an invariant principal connection.
\end{proposition}

Note that in term of the Maurer--Cartan form $\conn_{-}$ defined in~\eqref{eq:defmaurercartanminus}, \(\theta\) is defined as $\pairing{\theta}{X} \coloneqq -\pairing{\conn_{-}}{X}_g - g\inv \mul \pairing{\conn}{[X]}_{[g]} \mul g$.

\begin{proof}
Let us show that $\theta$ is indeed $\isoalg$-valued.
Define $\xi \coloneqq X \mul g\inv$,
and $\zeta \coloneqq \pairing{\conn}{[X]}_{[g]}$.
Notice that 
\begin{align}
	\label{eq:relthetaxizeta}
	\pairing{\theta}{X} = g\inv \mul (\xi - \zeta) \mul g
	.
\end{align}

The projected vector $[X]$ is  $[X] = \xi \act [g]$.
Now the consistency condition~\eqref{eq:conncons} reads
\(
	\zeta \act [g] = [X] = \xi \act [g],
\)
so, by multiplying on the left by $g\inv$ (which corresponds to pushing the vector forward to the identity), we obtain $g\inv\mul(\xi - \zeta)\mul g \mul \isogrp = 0$, which shows that \begin{align}g\inv\mul(\xi-\zeta)\mul g \in \isoalg,\end{align} which, along with~\eqref{eq:relthetaxizeta} implies $\pairing{\theta}{X} \in \isoalg$.

Now, choosing $X =  g \mul \xi$ for $\xi \in \isoalg$, we obtain $[X] = 0$, so $\pairing{\theta}{X} = g\inv \mul X = \xi$, and the consistency condition~\eqref{eq:princconnconscond} is thus fulfilled.

We obtain the $\isogrp$-equivariance by noticing that $[X\mul h] = [X]$, so
\begin{align}
	\pairing{\theta}{X \mul h}_{g\mul h} =
	h\inv \mul g\inv \mul X \mul h - h\inv \mul g\inv  \mul\pairing{\conn}{[X]}_{[g]}  \mul g\mul h
	= h\inv \mul \pairing{\theta}{X}_g \mul h
	.
\end{align}

Finally, we obtain the invariance of $\theta$ by using that $[\gb\mul X] = \Tan\gb \mul [X]$ and the equivariance of $\conn$:
\begin{align}
	\pairing{\theta}{\gb \mul X}_{\gb\mul g} &= g\inv\mul \gb \inv \mul \gb \mul X - g\inv\mul\gb\inv\mul \pairing{\conn}{\Tan\gb\mul [X]}_{\gb [g]}\mul\gb\mul g \\
	&= g\inv \mul X - g\inv\mul \pairing{\conn}{[X]}_{g} \mul g \\
	&= \pairing{\theta}{X}_{g}
	.
\end{align}
\end{proof}

In general, a connection on a principal $\isogrp$-bundle allows to \emph{lift} vector fields from the base to the manifold.
In our case, it means that one can lift vector fields on $\Man$ to vector fields on $\symgrp$.

The verification of the following proposition is straightforward.

\begin{proposition}
	Consider a vector field $f\in\diff$.
	Consider the lifted vector field $\lifted[+][\conn]$ defined in~\eqref{eq:deflifted}, i.e., by $\lifted[+][\conn](g) \coloneqq \pairing{\conn}{f}_{[g]}\mul g$.
	$\lifted[+]$ is then the horizontal lift of $f$ with respect to the invariant principal connection $\theta$ defined in \autoref{prop:princconn}.
\end{proposition}

\subsection{Flatness and Symmetry}

Since a connection  defined in \autoref{def:connection} can also be regarded as a principal connection (\autoref{sec:princconn}), it has a curvature.
It also has a torsion, as it is also an \emph{affine connection}~\cite{No54}.
We will not need the exact formulas for the torsion and curvature\footnote{The interested reader is referred to~\cite{No54}, or~\cite[\S\,X.2]{KoNo69}.}, and focus on whether the connection is flat (zero curvature), or symmetric (torsion free).

\begin{definition}
	We say that a connection (and its corresponding reductive structure $\redalg$) is \demph{symmetric} (or torsion free) if
	\begin{align}
		[\redalg, \redalg] \subset \isoalg
		.
	\end{align}
	We say that the connection is \demph{flat} (or has zero curvature) if $\redalg$ is a Lie subalgebra, that is
	\begin{align}
		[\redalg, \redalg] \subset \redalg
		.
	\end{align}
\end{definition}

The connection is thus flat and symmetric if and only if $\redalg$ is a trivial Lie algebra, i.e., $[\redalg, \redalg] = 0$.

In the presence of a symmetric connection, the homogeneous space is called a \demph{symmetric space}.
In particular, Cartan performed a classification of the symmetric spaces with compact isotropy group~\cite{Ca26}.
We refer to the monograph~\cite{He62} for further details.

\subsection{Existence of Connections}

We first give a general theorem of existence of reductive structure (and thus of connections)~\cite[\S\,X.2]{KoNo69}.
\begin{proposition}
\label{prop:connexistence}
	If the isotropy group $\isogrp$ is either compact, or connected and semi-simple, there exists a reductive structure.
\end{proposition}

In practice, the results of \autoref{prop:connexistence} or \autoref{prop:connequi} are of limited use, because it is preferable to have an explicit expression for the connection.
The following result, which proof is left to the reader, is used repeatedly to obtain tractable formulas in the examples of \autoref{sec:connexamples}, and to calculate the flatness or symmetry of the connection.

\begin{proposition}
\label{prop:formula}
	Suppose that a linear form $\conn \in \Omega^1(\Man, \symalg)$ is $\symgrp$-equivariant and consistent at the origin, i.e.,
	\(
		\pairing{\conn}{f}_{\origin} \act \origin = f(\origin)
	\).
	Then it is a connection corresponding to the reductive decomposition 
	\begin{align}
	\redalg = \conn(\Tan[\origin]\Man)
	.
	\end{align}
\end{proposition}

\section{Connection Examples}
\label{sec:connexamples}

We study connections in homogeneous spaces of interest in numerical analysis.
The situation varies a lot.
Some homogeneous spaces have no connection at all (\autoref{prop:fixedranknoconn}), some have only one (\autoref{prop:scalred}), some have infinitely many (\autoref{sec:cartanschouten}).
There is also a practical aspect, as the connection is not always available in closed form.
We refer to \autoref{tab:homogconn} for a summary of the study of the examples in this section.

\newcommand*\mmidrule{\arrayrulecolor{gray!30}\midrule\arrayrulecolor{black}}
\newcommand*\cmmidrule[1]{\arrayrulecolor{gray!30}\cmidrule{#1}\arrayrulecolor{black}}
\newcommand*\crossmark{\text{\sffamily x}}
\afterpage{
	\clearpage
\begin{landscape}
		\centering
	\begin{tabular}{llll|c|cc|l}
		\toprule
		\rowcolor{Linen}
		Manifold & Group & Isotropy & Action  & Connection & Sym. & Flat 
		&
		\S
		\\
		\midrule
		$\RR^d$ 
		&
		$\isogrp\ltimes \RR^d$ 
		&
		$\isogrp \subset \GL$ 
		&
		$h x + a$
		& $ \affalg{0}{\delta x}$& $\checkmark$ & $\checkmark$  
		&
		\ref{sec:affine}
		\\
		\midrule
		Stiefel 
		&
		$\SO[d]$
		&
		$\SO[d-k]$ 
		&
		$R \mul Q$ 
		&
		$ \begin{aligned} 
			&
			{\delta Q} \mul Q\tp\\-&Q \mul {\delta Q}\tp \\ -& Q \mul \delta Q\tp \mul Q \mul Q\tp \end{aligned}$
		& $\crossmark$ &  $\crossmark$ 
		&
		\ref{sec:stiefel}
		\\
		\cmmidrule{1-7}
		Sphere
		& 
		&
		$\SO[d-1]$
		&
		& $\begin{aligned}  {\delta Q} \mul Q\tp- Q \mul {\delta Q}\tp  \end{aligned}$
		& $\checkmark$ & $\crossmark$
		&
		\\
		\midrule
		Isospectral 
		&
		$\SO[\abs{\prt}]$ 
		&
		$\SO[\prt]$ 
		&
		$R \mul P \mul R\tp$ 
		& No formula
		& $\crossmark$
		& $\crossmark$
		&
		\ref{sec:isospectral}
		\\
		\cmmidrule{1-3}
		\cmmidrule{6-8}
		$\SO[\abs{\prt}]$
		&
		&
		$\one$ 
		&
		&
		& $\crossmark$ & $\checkmark$ 
		&
		\ref{sec:isoflat}
		\\
		\mmidrule
		Grassmannian
		&
		&
		$\group{S}\paren[\big]{\OO[k] \times \OO[d-k]}$ 
		&
		& $ \delta P \mul  P -  P \mul \delta P$ 
		& $\checkmark$ & $\crossmark$ 
		&
		\ref{sec:grassmannian}
		\\
		Projective
		&
		&
		$\SO[d-1]$ 
		&
		&
		&
		&
		\\
		\midrule
		SPD Matrices
		&
		$\GL$
		& 
		$\OO$
		& $A \mul P \mul A\tp$
		& $ \frac{1}{2}\delta P \mul P\inv$
		& $\checkmark$ & $\crossmark$ 
		&
		\ref{sec:polar}
		\\
		\midrule
	$\symgrp$
	&
	$\symgrp$
	&
	$\one$
	& $\bar{g} \mul g$ & $ \conn_{+} = \phantom{-}\delta g \mul g \inv$ & $\crossmark$ & $\checkmark$ 
	&
	\ref{sec:liegroupact}
	\\
	\cmmidrule{4-5}
	 &  &  & $g \mul \bar{g}\inv$ & $ \conn_{-} = -g \inv \mul \delta g$ & $\crossmark$ & $\checkmark$ 
	 &
	 \\
		\midrule
		$\symgrp$ & $\symgrp\times\symgrp$ & $\symgrp$ & $g_1 \mul g \mul g_2\inv$ & $ \paren{ \conn_{+}, \conn_{-}}/2$& $\checkmark$ & $\crossmark$ 
		&
		\ref{sec:cartanschouten}
		\\
		\cmmidrule{5-7}
		& & & & $(0,\conn_{-})$ & $\crossmark$ & $\checkmark$ 
		&
		\\
		\cmmidrule{5-7}
		& & & & $(\conn_{+},0) $ & $\crossmark$ & $\checkmark$ 
		&
		\\
	 \midrule
	\begin{tabular}{l}Fixed\\ rank\\ matrices\end{tabular}
	&
	$\symgrp = \GL[m]\times\GL[n]$
	&
		$\paren*{
	\begin{bmatrix}
	A_{1} & 0	 \\
	A_{2} & C
	\end{bmatrix}
	,
	\begin{bmatrix}
		B_{1} & 0 \\
		 B_{2} & C
	\end{bmatrix}
	} \in \symgrp$
	&
	$A \mul M \mul B\inv$ &
	$\emptyset$ 
	&
	$\emptyset$ &  $\emptyset$
	&
	\ref{sec:fixedrank}
	\\
		\bottomrule
	\end{tabular}
	\captionof{table}{
		A summary of the connections (or absence thereof) described in this paper.
	We refer to the respective section in the last column for the notations used in this table.
	}
	\label{tab:homogconn}
\end{landscape}
	\clearpage
}

\subsection{Lie Group as Homogeneous Spaces}

A Lie group can be regarded as a homogeneous space in at least three ways, which we now describe.

\subsubsection{Left and Right Actions}
\label{sec:liegroupact}

The Lie group acts transitively on itself by \emph{left multiplication}:
\begin{align}
g \act g' \coloneqq g \mul g' \qquad
g,g' \in \symgrp
.
\end{align}
The isotropy group in that case is trivial, so the only reductive  structure is $\redalg = \symalg$.
The corresponding flat  connection is the \demph{Maurer--Cartan form}~\cite{Ca10}.
It is defined by
\begin{align}
	\label{eq:defmaurercartan}
	\pairing{\conn_+}{\delta g}_g \coloneqq \delta g \mul g\inv
	.
\end{align}
Note that this connection is symmetric if and only if $\symgrp$ is commutative.

Of course, there is also the corresponding right multiplication action
\begin{align}
g \act g' \coloneqq g'  \mul g\inv
,
\end{align}
for which the flat  connection is
\begin{align}
	\label{eq:defmaurercartanminus}
	\pairing{\conn_{-}}{\delta g}_g \coloneqq -g\inv \mul \delta g
	.
\end{align}

\subsubsection{Cartan--Schouten Action}
\label{sec:cartanschouten}

There is another way in which $\symgrp$ can be a homogeneous manifold~\cite{CaSc26}, \cite[\S\,X.2]{KoNo69}. 
The symmetry group is now $\symgrp\times\symgrp$, acting on the manifold $\Man = \symgrp$ by
\begin{align}
	\label{eq:actcartanschouten}
(g_1,g_2) \act g \coloneqq g_1 \mul g \mul g_2\inv
.
\end{align}
We choose the origin at the group identity.
The isotropy group is then
\begin{align}
	\isogrp = \setc{(g,g)}{g \in \symgrp} \equiv \symgrp
	,
\end{align}
with corresponding Lie algebra
\begin{align}
	\isoalg = \setc{(\xi,\xi)}{\xi \in \symalg}
	.
\end{align}

Define (with a slight abuse of the notation~\eqref{eq:defmaurercartan}) the $\symalg \times \symalg$-valued one-form
\begin{align}
	\pairing{\conn_+}{\delta g} \coloneqq (\delta g \mul g \inv, 0)
	.
\end{align}

The infinitesimal action of $(\xi,\zeta)\in\symalg\times\symalg$ on $g\in\symgrp$ is given by $\xi \mul g - g \mul \zeta$, so we obtain consistency at the origin.
For equivariance, we check that
\begin{align}
	\pairing{\conn_+}{g_1 \mul \delta g \mul g_2\inv}_{g_1 \mul g \mul g_2\inv} &= (g_1 \mul \delta g \mul g_2\inv (g_1\mul g \mul g_2\inv)\inv, 0) \\
	&= (g_1 \mul \delta g \mul g\inv \mul g_1\inv, 0) \\
	&= (g_1,g_2) \act (\delta g \mul g\inv, 0)
	.
\end{align}
This shows, using \autoref{prop:formula}, that $\conn_{+}$ is a connection and that
\(
	\redalg_+ \coloneqq (\symalg , 0)
\)
is a reductive structure.
The connection $\conn_{+}$ is flat, because $\redalg_{+}$ is isomorphic to the Lie algebra $\symalg$.
It is thus symmetric if and only if $\symgrp$ is commutative.

There is also a corresponding connection
\begin{align}
	\pairing{\conn_-}{\delta g} \coloneqq (0,  -g \inv \mul \delta g)
\end{align}
associated to the reductive structure $\redalg_{-} = (0, \symalg)$.

As the set of connections is an affine space, the \emph{mean value} \begin{align}\conn_0 = (\conn_+ + \conn_-)/2\end{align} of those two connections is also a connection.
It also happens to be a \emph{symmetric} connection.
Indeed, the associated subspace is \(\redalg_0 = \conn_0(\Tan[\origin] \symgrp) = \setc{(\xi, -\xi)}{\xi \in \symalg}\), and it is easy to check that \([\redalg_0, \redalg_0] \subset \isoalg\), showing that the connection is symmetric.
The connection \(\conn_0\) is thus flat if and only if \(\symgrp\) is commutative.

We have seen that, in this case, there are at least two reductive structures (so the dimension of the reductive structure space is at least two), but there may be more~\cite{Be72}.

\subsection{Affine spaces}
\label{sec:affine}

We define an \demph{affine space} as the manifold $\Man \equiv \RR^d$ and the group $\symgrp = \isogrp\ltimes\RR^d$, where $\isogrp$ is a subgroup of $\GL$.
We represent an element of $\symgrp$ by
\begin{align}
g = \affgrp{h}{a}
,
\qquad
h\in\isogrp
,
\quad
a \in \RR^d
,
\end{align}
acting on an element
\begin{align}
\begin{bmatrix}
x \\ 1	
\end{bmatrix}
\end{align}
by matrix multiplication.
The action is thus
\begin{align}
g\act x = h\mul x + a
.
\end{align}
We choose the origin at zero, i.e., $\origin = 0$, and the isotropy group is then $\isogrp$.

On an affine space, there is always an ``obvious'' connection given by translations.

\begin{proposition}
\label{prop:affineconnection}
	In any affine space \(\isogrp \ltimes \RR^d\) there is a connection given by
	\begin{align}
		\label{eq:affconn}
		\pairing{\conn}{\delta x}_x = \begin{bmatrix}
			0 & \delta x \\ 0 & 0
		\end{bmatrix}
		.
	\end{align}
\end{proposition}

We now study under which conditions the connection~\eqref{eq:affconn} is the only possible one.

\begin{lemma}
	Consider the semidirect product \(\symgrp = \isogrp \ltimes \RR^d\), where $\isogrp \subset \GL$.
	For any linear map
	\begin{align}
		\alpha \colon \RR^d \to \isoalg	
	\end{align}
	such that
	\begin{align}
		\label{eq:cocycle}
	[A, \alpha(x)] = \alpha(A \mul x)
	,
	\qquad
	A \in \isoalg
	,
	\quad
	x \in \RR^d
	,
	\end{align}
	the subspace
	\begin{align}
		\label{eq:redalgcocycle}
		\redalg = \setc*{\begin{bmatrix}\alpha(x)& x \\ 0&0\end{bmatrix}}{x \in \RR^d}
	\end{align}
	is a reductive structure.
	Conversely, for any reductive structure, there is a linear map $\alpha$ fulfilling~\eqref{eq:cocycle}, such that $\redalg$ is defined by~\eqref{eq:redalgcocycle}.
\end{lemma}
\begin{proof}
	An element of $\isoalg$ can be written as
	\begin{align}
		\xi = \begin{bmatrix}
	A & 0 \\ 0 & 0
		\end{bmatrix}
		,
		\qquad
		A \in \isoalg
		.
	\end{align}
	Suppose that $\redalg$ is a reductive structure.
	It is parameterised by $x \in \RR^d$ as
	\begin{align}
		\zeta = \begin{bmatrix}
		\alpha(x) & x \\ 0 & 0
		\end{bmatrix}
	\end{align}
	for some linear function $\alpha$.
	We compute the commutator
	\begin{align}
	[\xi, \zeta] = \begin{bmatrix}
		[A, \alpha(x)] & A \mul x \\ 0 & 0
	\end{bmatrix}
	.
	\end{align}
	As $\redalg$ is reductive, we obtain $[A, \alpha(x)] = \alpha(A \mul x)$.
	It is straightforward to check the opposite statement.
\end{proof}

We now obtain the following uniqueness result if the isotropy group contains scalings.
Note how the presence of scalings in the isotropy group also simplified the study of $\isogrp$-invariant spaces in~\cite[\S\,6]{MKVe13}.

\begin{proposition}
\label{prop:scalred}
	If the isotropy group contains scalings, then the only reductive structure is $\redalg = \RR^d$.
	The only connection is then given by~\eqref{eq:affconn}.
\end{proposition}
\begin{proof}
	If $\gl[1] \subset \isoalg$, it means that $\Id \in \isoalg$.
	By using~\eqref{eq:cocycle} with $A = \Id$, we obtain that $\alpha(x) = 0$, so $\redalg = \RR^d$.
\end{proof}

\subsection{Stiefel Manifolds and spheres}
\label{sec:stiefel}

For integers $k \leq n$, the \demph{Stiefel manifold} is the set of $n \times k$ matrices $Q$ with orthogonal columns, i.e.,
\begin{align}
	\label{eq:defstiefel}
Q\tp \mul Q = 1
.
\end{align}
An element $R \in \SO[n]$ acts on $Q$ by 
\begin{align}
	R\act Q \coloneqq R \mul Q
	.
\end{align}
This action is transitive.
If we define the origin to be the matrix (in block notation)
\begin{align}
	\label{eq:stiefelorigin}
	Q_0 \coloneqq \begin{bmatrix}
		0 \\ \one
	\end{bmatrix}
\end{align}
then the isotropy group consists of the matrices
\begin{align}
	\isogrp = \begin{bmatrix}
		R & 0 \\ 0 & \one
	\end{bmatrix} 
	,
	\qquad
	R \in \SO[n-k]
	.
\end{align}
So, as is customary, we identify $\isogrp \equiv \SO[n-k]$.
The Lie algebra $\symalg$ is $\so[n]$, the space of skew symmetric matrices.

A candidate for a reductive decomposition is the space $\redalg$ consisting of matrices of the form
\begin{align}
	\label{eq:defstiefelred}
	\begin{bmatrix}
		0 & W \\ -W\tp & \Omega
	\end{bmatrix}
	,
	\qquad
	W \in \RR^{(n-k)\times k} 
	,
	\quad
	\Omega \in \so[k]
	.
\end{align}

\begin{proposition}
The space $\redalg$ is a reductive structure.
The associated connection defined at a vector $\delta Q$ on a point $Q$ is
\begin{align}
	\label{eq:stiefelconn}
	\pairing{\conn}{\delta Q}_{Q} =   {\delta Q} \mul Q\tp -Q \mul {\delta Q}\tp - Q \mul \delta Q\tp \mul Q \mul Q\tp
	.
\end{align}
The connection is symmetric if and only if $k=1$, i.e., if the Stiefel manifold is a sphere, in which case the connection simplifies into
\begin{align}
	\pairing{\conn}{\delta Q}_{Q} =   {\delta Q} \mul Q\tp -Q \mul {\delta Q}\tp 
	.
\end{align}
\end{proposition}
\begin{proof}
	By differentiating~\eqref{eq:defstiefel} we obtain that a tangent vector $\delta Q$ at $Q$ satisfies
	\begin{align}
		\label{eq:stiefelortho}
		\delta Q\tp \mul  Q  + Q\tp \mul \delta Q = 0
		.
	\end{align}
	This shows that $\conn$ takes its value in $\symalg = \so[n]$.

	Recalling the action of $\SO[n]$ on the manifold by left multiplication, the lifted action on tangent vectors is also by left multiplication (as the action is linear).
	Now, for any matrix $R\in\SO[n]$, we have
	\begin{align}
		\pairing{\conn}{R \act \delta Q}_{R \act Q} = R \mul \pairing{\conn}{\delta Q}_Q \mul R\tp = R \act \pairing{\conn}{\delta Q}_Q
		,
	\end{align}
	which shows the equivariance of $\conn$.

	At the origin $Q_0$, 
	the infinitesimal action of 
\begin{align}
		\xi = \begin{bmatrix}
			\star & W \\
		W\tp & \Omega
	\end{bmatrix}
	\in\symalg
\end{align}
is 
	\begin{align}
		\xi \act Q_0 = \begin{bmatrix}
			W \\ \Omega
		\end{bmatrix}
		.
	\end{align}
	The connection sends that vector to
	\begin{align}
		 \begin{bmatrix}
		0 & W \\
		W\tp & \Omega
	\end{bmatrix}
	,
	\end{align}
	from which consistency follows.

	Finally, the image $\redalg = \conn(\Tan[\origin]\Man)$ consists of matrices of the form~\eqref{eq:defstiefelred}.
	We conclude using \autoref{prop:formula}.

	Finally, computing an extra diagonal term of the commutator between two elements of $\redalg$ shows that it is zero if and only if the lower right term is zero, i.e., $\Omega = 0$.
	This happens only if $k=1$.
Finally, in this case, the orthogonality condition~\eqref{eq:stiefelortho} becomes $ \delta Q\tp \mul  Q = 0$, which accounts for the simplification in the connection formula.
\end{proof}

Note that the expression~\eqref{eq:stiefelconn} is also obtained in~\cite{CeOw03}, although the authors do not mention the equivariance of that connection.

We also notice that if $Q'$ is orthogonal to both $Q$ and $\delta Q$, then the vector field at $Q'$ is zero, that is
\begin{align}
	\pairing{\conn}{\delta Q}_{Q} \act Q' = 0
	.
\end{align}
This is especially intuitive on \emph{spheres}, where it means that the (generalised) axis of rotation is orthogonal to the plane spanned by $Q$ and $\delta Q$.

\subsection{Isospectral Flows, Lax Pairs, Grassmannians and Projective Spaces}
\label{sec:isospectral}

An isospectral manifold is the space of symmetric matrices with a prescribed spectrum~\cite{CaIsZa97}.
As we shall see, they contain as special cases Grassmannians (and thus projective spaces), as well as principal homogeneous spaces for the rotation group, in which case they are often expressed in the form of a Lax pair.

In order to define the isospectral manifold, we first define a {partition} of an integer $d$, representing eigenvalue multiplicities.
We define a \demph{partition} to be a function $\prt \colon \NN \to \NN$, which is \emph{non increasing}, and eventually zero.
The \demph{size} of the partition is
\begin{align}
	\abs{\prt} = \sum_{i=0}^{\infty} \prt_i
\end{align}
which is a finite integer.

We define the \demph{isospectral manifold} $\Man$ associated to a partition $\prt$ and a sequence of (necessarily distinct) real eigenvalues $\lambda_i$ as
\begin{align}
	\Man \coloneqq \setc[\Big]{P \in \RR^{\abs{\prt} \times \abs{\prt}}}{P = P\tp; \quad \text{$P$ has eigenvalue $\lambda_i$ with multiplicity $\prt_i$}}
	.
\end{align}

The \demph{length} of the partition indicates the number of distinct eigenvalues.
It is defined as the number of non zero elements in $\prt$, i.e.
\(
	 \#\setc{i\in\NN}{\prt_i \neq 0}
	.
\)

Examples of partitions are $\kappa = (1)$ of length and size one, $\kappa = (3,2)$ of size five and length two.
The partitions associated to a \demph{Grassmannian} have length two, such as $(m,n)$, with two arbitrary positive integers $m$, $n$ such that $m \geq n$.
The \demph{projective space} case corresponds to the special case $n = 1$.

To a given partition $\prt$ we associate a block matrix representation where the block $i,j$ has size $\prt_i \times \prt_j$.
Note that the size of the partition gives the size of the matrices, whereas the length of the partition is the number of blocks.

	In what follows, we exclude the trivial case of $\prt$ having length one, in which case the manifold $\Man$ reduces to one point.

The group $\SO[\abs{\prt}]$ acts on $P$ by adjunction
\begin{align}
	\label{eq:actiso}
R \act P = R \mul P \mul R\tp
.
\end{align}
We can define the origin $\Delta$ to be the diagonal matrix with eigenvalues $\lambda_i$ with multiplicity $\prt_i$.
As symmetric matrices are diagonalisable with orthogonal matrices, we obtain that any matrix in $P \in \Man$ can be written as $P = R\mul \Delta \mul R\tp$ for some $R\in\SO[\abs{\prt}]$.
This shows that $\SO[\abs{\prt}]$ acts transitively on $\Man$.
The isotropy subgroup $\isogrp$ is then the set of block matrices of the form
\begin{align}
	\label{eq:typicalspeciso}
\begin{bmatrix}
	R_1 & & \\
	& R_2 &\\
	&& \ddots
\end{bmatrix}
\end{align}
such that $R_i \in \OO[\prt_i]$, and such that the determinant is one.
We thus denote the isotropy group by
\begin{align}
	\isogrp = \SO[\prt] \coloneqq \group{S}\paren[\Big]{\prod_{i\in\NN} \OO[\prt_i]}
	.
\end{align}
Note that, since the isotropy group does not depend on the eigenvalues, but only on their multiplicity, this shows that all the isospectral manifolds with the same partition (but possibly different eigenvalues) are isomorphic.

A good candidate for the reductive structure $\redalg$ is the subspace of $\so[\abs{\prt}]$ which is zero on the block diagonal:
\begin{align}
\begin{bmatrix}
	0 && \\
	& 0 & \\
	&& \ddots \\
\end{bmatrix}
\in \so[\abs{\prt}]
.
\end{align}

\begin{proposition}
\label{prop:isospectral}
	The space $\redalg$ is a reductive structure.
	It further induces
	\begin{enumerate}[label=\upshape(\roman*)]
		\item a \emph{flat} connection if and only if $\prt = (1,1,\ldots,1)$, in which case it is a principal homogeneous space for $\SO[d]$.
		\item a \emph{symmetric} connection if and only if $\prt = (k,d-k)$, in which case $\Man$ is isomorphic to a Grassmannian.
			The corresponding connection is given explicitly by
			\begin{align}
				\label{eq:grassconn}
				\pairing{\conn}{\delta P}_P \coloneqq \frac{1}{\paren{\Delta \lambda}^2}\paren{\delta P \mul P - P \mul \delta P}
				,
			\end{align}
			where $\Delta \lambda$ is the difference between the two eigenvalues at hand.
	\end{enumerate}
\end{proposition}
\newcommand*\mb{\bar{m}}
\begin{proof}
	We denote a generic element of $h \in \isogrp$ which is in the form~\eqref{eq:typicalspeciso}.
	A typical element of $\redalg$ can be written, using the same block conventions, 
	\begin{align}
	m = \begin{bmatrix}
		0 & m_{12} & \ldots \\
		m_{21} & 0 & \ldots \\
		\vdots &&
	\end{bmatrix}
	\end{align}
	with $m_{ij} + m_{ji}\tp = 0$.
	\begin{enumerate}
		\item
			First, a calculation shows that a product of an element $h\in\isoalg$ and an element $m\in\redalg$ is such that $h \mul m \in \redalg$, and $m \mul h \in \redalg$, so, in particular, the commutator $[h,m] \in \redalg$, which shows that $\redalg$ is a reductive structure.
		\item
The first extra diagonal term, of index $(2,1)$ in the product $m\mul \mb$ of two elements $m,\mb \in \redalg$ is
\begin{align}
	\underbrace{m_{11}}_{=0} \mul \mb_{12} + m_{12} \mul \underbrace{\mb_{22}}_{=0} + m_{13} \mul \mb_{32} + \cdots
\end{align}
This shows that if the length of the partition $\prt$ is two, then the product is zero at the extra diagonal, and in particular, the commutator is as well, so the connection is symmetric.
If the length is greater than two, the commutator contains at least the term
\begin{align}
m_{13} \mul \mb_{32} - \mb_{13} \mul m_{32}
,
\end{align}
so it is always possible to choose, for instance the blocks $m_{13}$ and $\mb_{32}$, such that this term does not vanishes.
This shows that the connection is not symmetric.
\item
An element in the block diagonal of the commutator $[m,\mb]$ can be written as
\begin{align}
	\sum_k -m_{ik} \mul \mb_{ik}\tp + \mb_{ik} \mul m_{ik}\tp
\end{align}
If all the blocks have size one, then this vanishes, showing that if all the eigenvalues have multiplicity one, the connection is flat.
Note that, in that case, the isotropy group is trivial, which shows that $\Man$ is a principal homogeneous space for $\symgrp = \SO[\abs{\prt}]$.
If a block of index $k$ has size greater than one, then one can choose $m_{ik}$ and $\mb_{ik}$ so that $-m_{ik}\mul\mb_{ik}\tp + \mb_{ik}\mul m_{ik}\tp \neq 0$.
\item
	The equivariance property of~\eqref{eq:grassconn} is straightforward to check.
	It now suffices to check consistency at the origin.
	Recall that we are considering the case
	\begin{align}
		\Delta = \begin{bmatrix}
			\lambda_1 & 0 \\ 0 & \lambda_2
		\end{bmatrix}
	\end{align}
		with the same block-matrix conventions as before, i.e., with respect to the partition $(k, d-k)$.
		From~\eqref{eq:actiso}, we obtain that the infinitesimal action of $\xi \in \so[d]$ at a point $P$ is given by
		\begin{align}
			\label{eq:actisoinf}
			\xi \act P = \xi \mul P - P \mul \xi
			.
		\end{align}
	A direct calculation shows that the commutator of a matrix $A$ with block entries $a_{ij}$ and $\Delta$ has block coefficients $\bracket{A, \Delta}_{ij} = a_{ij} (\lambda_j - \lambda_i)$.
	Now, an element of $\redalg$ takes the form
	\begin{align}
	m = \begin{bmatrix}
		0 & m_{12} \\ -m_{12}\tp & 0
	\end{bmatrix}
	,
	\end{align}
	so $[m,\Delta]$ is a symmetric block matrix 
	\begin{align}
		\delta \Delta = \begin{bmatrix}
		0 & m_{12}(\lambda_2 - \lambda_1) \\
		m_{12}\tp (\lambda_2 - \lambda_1) & 0
	\end{bmatrix}
	.
	\end{align}
	Applying the commutator $[\delta \Delta, \Delta]$ gives the matrix with element $m_{ij}\paren{\lambda_2 - \lambda_1}^2$.
	We thus obtain $[\delta \Delta, \Delta] = \paren{\lambda_2 - \lambda_1}^2 m$, which finishes the proof.
	\end{enumerate}
\end{proof}

\subsubsection{The symmetric case: Grassmannians}
\label{sec:grassmannian}

Note that Grassmann manifolds are generally defined as the set of orthogonal projectors on subspaces of dimension $k$~\cite[\S\,IV.9.2]{HaLuWa06}.
It means that the eigenvalues of the corresponding isospectral manifold are zero and one, so~\eqref{eq:grassconn} simplifies into
\begin{align}
			\pairing{\conn}{\delta P}_P = {\delta P \mul P - P \mul \delta P}
			.
\end{align}

From an element $Q \in \RR^{n\times k}$ in a Stiefel manifold (see \autoref{sec:stiefel}), one obtains an element $P$ in its corresponding Grassmann manifold by
\begin{align}
P = Q \mul Q\tp
.
\end{align}

Grassman manifolds can thus be regarded as the base manifold of a Stiefel manifold considered as a principal bundle.
This might be exploited by constructing descending integrators on the Stiefel manifold, instead of the Grassmann manifolds, as we discuss in the conclusion.

\subsubsection{The flat case: integrable Lax pairs}
\label{sec:isoflat}

The differential equation on an isospectral manifold is often given directly by the infinitesimal action of $\so$, in which case it is called a \emph{Lax pair}.
Indeed, from~\eqref{eq:actisoinf}, we have
\begin{align}
	P' = \xi(P) \act P = \xi(P) \mul P - P \mul \xi(P)
	.
\end{align}

A particularly interesting case of Lax Pair system is the \emph{integrable case}, when all the eigenvalues are distinct, as for instance, the Toda lattice~\cite[\S\,X.1.5]{HaLuWa06}.
As we saw in \autoref{prop:isospectral}, this corresponds to the flat case.
The isospectral manifold is then a principal homogeneous space, so there is no ambiguity in the isotropy, and the only possibility is thus \begin{align}\pairing{\conn}{\delta P} = \xi(P),\end{align}
so the Lax pair already gives the solution of the connection problem.
This is essentially the method considered in~\cite{CaIsZa97}, and we see now that this method is equivariant.
We refer to that aforementioned paper and to~\cite[\S\,X.1.5]{HaLuWa06} for further insights into Lax pairs and integrable systems.

\subsection{Polar Decompositions}
\label{sec:polar}

There is a homogeneous manifold naturally associated with the polar decomposition of an invertible matrix in an orthogonal and a positive-definite matrix~\cite[\S\,3.2]{Ga14}.
The group is
\begin{align}
	\symgrp = \GL
\end{align}
and the manifold $\Man$ is the set of \demph{symmetric positive definite matrices}:
\begin{align}
	\Man = \setc[\Big]{P \in \RR^{d\times d}}{P = P\tp \qquad x\tp \mul P \mul x > 0 \quad\forall x \in \RR^d \setminus{0}} 
	.
\end{align}
The action of $A\in\GL$ on $P\in\Man$ is given by
\begin{align}
A \act P \coloneqq A \mul P \mul  A\tp
.
\end{align}
It is a transitive action as any positive definite matrix $P$ can be written as
\(
P = A \mul A\tp
\)
for some matrix $A \in \GL$.
We choose the origin
\begin{align}
P_0 \coloneqq \one
,
\end{align}
which gives the isotropy group
\begin{align}
	\isogrp = \OO
	.
\end{align}

We have the following connection.
\begin{proposition}
	The space of symmetric matrices $\redalg \coloneqq \setc{\xi\in \gl}{\xi = \xi\tp}$ is a symmetric reductive structure.
	The corresponding symmetric connection is defined by
  \begin{equation}
		\label{eq:polarconn}
    \pairing{\conn}{\delta P} \coloneqq \frac{1}{2}\delta P \mul P\inv
    .
  \end{equation}
\end{proposition}
\begin{proof}
	The commutator of an antisymmetric matrix and a symmetric matrix is a symmetric matrix, so $[\isoalg,\redalg] \subset \redalg$.
	Moreover, we have $\symalg = \isoalg \oplus \redalg$, so $\redalg$ is a reductive structure.
	The commutator of two symmetric matrices is an skew symmetric matrix, i.e., $[\redalg, \redalg] \subset \isoalg$ and the connection is thus symmetric.

  We check that the formula~\eqref{eq:polarconn} is equivariant:
  $\pairing{\conn}{A \mul (\delta P) \mul A\trans} = \frac{1}{2} A \mul (\delta P) A \tp (A \mul P \mul A\tp)\inv
  = \frac{1}{2}A \mul(\delta P) \mul P \mul A\inv = A \act \pairing{\conn}{\delta P}$.

  Moreover, at the origin, $\pairing{\conn}{X} \act \origin = \pairing{\conn}{X} + \pairing{\conn}{X}\tp = (X + X\tp)/2 = X$, so the connection formula is consistent,
  because $X$, being tangent to a manifold of symmetric matrices, has to be symmetric.
\end{proof}

\subsection{Matrices of fixed rank}
\label{sec:fixedrank}

The space of matrices of fixed rank is considered in~\cite[\S\,IV.9.3]{HaLuWa06} and plays a fundamental role in low-rank approximations, and model reduction.
We show how this space has a natural structure of homogeneous space, and that it lacks reductive structures, and thus, connections.

For integers $m$, $n$, and $k$, we define the manifold $\Man$ of $m \times n$ matrices of rank $k$.
An element $(A,B) \in \GL[m] \times \GL[n]$ acts on such a matrix $M$ by
\begin{align}
	\label{eq:actrankk}
(A,B) \act M = A \mul M \mul B\inv
,
\end{align}
and the action is transitive on $\Man$.

Let us choose the origin at the $m\times n$ matrix $M_0$ of rank $k$ defined by
\begin{align}
	M_0 \coloneqq
\begin{bmatrix}
	0 & 0 \\ 0 & \one
\end{bmatrix}
,
\end{align}
where $\one$ denotes here the identity matrix of size $k$.
We assume the relevant block decomposition for the matrices $A$, $B$ and $M$ in the remaining of this section.

A calculation shows that the isotropy group consists of pairs of matrices of the form
\begin{align}
	\paren*{
\begin{bmatrix}
A_{1} & 0	 \\
A_{2} & C
\end{bmatrix}
,
\begin{bmatrix}
	B_{1} & 0 \\
 	B_{2} & C
\end{bmatrix}
} \in \GL[m] \times \GL[n]
.
\end{align}

\begin{proposition}
\label{prop:fixedranknoconn}
	The homogeneous space of $m\times n$ matrices of rank $k$  has a connection if and only if $m = n = k$, in which case it is the Cartan--Schouten manifold (see \autoref{sec:cartanschouten}) for $\GL[k]$.
\end{proposition}
\newcommand*\baone{\overline{a_1}}
\newcommand*\batwo{\overline{a_2}}
\newcommand*\bbone{\overline{b_1}}
\newcommand*\bbtwo{\overline{b_2}}
\newcommand*\bcc{\overline{c}}
\begin{proof}
	First, we notice that if $m = n = k$, then the rank $k$ matrices are simply the invertible matrices of size $k$, and the action~\eqref{eq:actrankk} is the Cartan--Schouten action~\eqref{eq:actcartanschouten}.

	We now assume that there exists a reductive structure, and proceed to show that it implies $m = n = k$.

	A reductive structure $\redalg \subset \symalg$ is parameterised by matrices $\alpha$, $\beta$, $\gamma$, and linear maps $\baone$, $\batwo$, $\bbone$, $\bbtwo$ and $\bcc$, all depending linearly on $\alpha$, $\beta$, $\gamma$, such that an element $m \in \redalg$ takes the form
	\begin{align}
		m = \paren*{
	\begin{bmatrix}
		\baone & \alpha \\
		\batwo & \bcc - \gamma
	\end{bmatrix}
	,
	\begin{bmatrix}
		\bbone & \beta \\
		\bbtwo & \bcc + \gamma
	\end{bmatrix}
}
.
	\end{align}
	We denote the infinitesimal adjoint action of an element $\xi \in \isoalg$ by
	\begin{align}
	\xi \act m \coloneqq [\xi, m] = \xi \mul m - m \mul \xi
	.
	\end{align}
	\begin{enumerate}
		\item Choose
			\begin{align}
				\xi = \paren*{
				\begin{bmatrix}
					1 & 0 \\ 0 & 0
				\end{bmatrix}
				,
				\begin{bmatrix}
					1 & 0 \\ 0 & 0
				\end{bmatrix}
			} \in \isoalg
			.
			\end{align}
			We obtain that
			\begin{align}
				\xi \act m = 
				\paren*{
\begin{bmatrix}
	0 & \alpha \\ -\batwo & 0
\end{bmatrix}				
,
\begin{bmatrix}
	0 & \beta \\ -\bbtwo & 0
\end{bmatrix}				
				}
				,
			\end{align}
			so we obtain
			\begin{align}
				\batwo(\alpha, \beta, 0) = - \batwo(\alpha, \beta, \gamma)
				.
			\end{align}
			By choosing $\gamma = 0$, we first obtain $\batwo(\alpha,  \beta, 0) = 0$, and thus also $\batwo(\alpha, \beta, \gamma) = 0$ for all values of the parameters $\alpha$, $\beta$, $\gamma$.
		\item
			Choose now an arbitrary element
			\begin{align}
			m = \paren*{
\begin{bmatrix}
	\baone & \alpha \\
	0 & \tilde{c}
\end{bmatrix}
, 
\star
			}
			\in \redalg
			,
			\end{align} 
			where $\star$ denotes here and in the sequel an arbitrary element.
			Choose an element
			\begin{align}
				\xi = \paren*{
\begin{bmatrix}
	0 & 0 \\ A & 0
\end{bmatrix}
, \star
				} \in \isoalg
				.
			\end{align}
			We obtain
\begin{align}
	\xi \act m = \paren*{
\begin{bmatrix}
	-\alpha \mul A & 0 \\
	A  \mul\baone - \tilde{c}  \mul A & A \mul \alpha
\end{bmatrix}, \star
	}
	,
\end{align}
so we obtain the condition
\begin{align}
A  \mul\baone - \tilde{c}  \mul A = 0
.
\end{align}
Now, the element $\xi\act m$ also belongs to $\redalg$, so the condition $\xi\act (\xi \act m) \in \redalg$ gives
\begin{align}
A \mul (-\alpha \mul A) - A \mul \alpha \mul A = 0
.
\end{align}
\item
	We choose in particular $A = \alpha\tp$, and we get
	\begin{align}
		\label{eq:alphathree}
		\alpha\tp \mul \alpha \mul \alpha\tp = 0
		.
	\end{align}
\item
	Now, the parameter $\alpha$ is arbitrary, so if $m > k$, we can choose it such that either $\alpha \mul \alpha\tp$ or $\alpha\tp \mul \alpha$ is the identity matrix, but~\eqref{eq:alphathree} then gives $\alpha = 0$, which is a contradiction.
	We conclude that $m = k$.
	A similar reasoning for $\beta$  would give $n = k$.
	\end{enumerate}
\end{proof}

The following example of a nonreductive space, mentioned in~\cite[\S\,X.2]{KoNo69},
\begin{align}
	\symgrp =  \group{SL}(2, \RR),
	\qquad
	\isogrp = \setc*{\begin{bmatrix}
		1 & \lambda \\ 0 & 1
	\end{bmatrix}}{\lambda \in \RR}
,
\end{align}
corresponds to $2\times 1$ matrices which are nonzero (i.e., which have rank one), so it is the case $m=2$, $n=1$ and $k=1$ in \autoref{prop:fixedranknoconn}.

\section{Conclusion and Open Problems}
\label{sec:conclusion}

The main message of this paper is that equivariant isotropy map allow to construct equivariant homogeneous space integrators from Lie group integrators (skeletons).
Moreover, when that equivariant isotropy map is of order zero, then it is equivalent to a reductive structure or an invariant principal connection, both of which are standard tools in differential geometry~\cite[\S\,X.2]{KoNo69}.

We examine the consequences of using a connection with a skeleton.

The first consequence is that some homogeneous spaces currently lack (equivariant) integrators.
A fundamental such example, as we showed in \autoref{sec:fixedrank}, is the homogeneous space of matrices of fixed rank.

Another consequence is that, as the connection takes its values in $\redalg$,  one only need compute exponentials of elements of $\redalg$.
For instance in the affine case, only translations are needed.
Note that in general, the computation of exponentials is still the most difficult part in the implementation, and we refer to~\cite[\S\,8]{IsMKNoZa00} for some possible solutions.

In a similar vein, note that if the motion map is the exponential, one need to know one exponential on each $\isogrp$-orbit since
\begin{align}
	\exp(h \act \xi) = h \act \exp(\xi) \qquad h\in\isogrp \quad \xi \in \redalg
	.
\end{align}
For instance, on a sphere, all the ``great circle'' are in a sense, the same, as there is only one $\isogrp$-orbit.

Let us mention some open questions to be investigated in future work.
\begin{itemize}
	\item
		It is not clear yet whether the isotropy map has to be equivariant for the integrator to be.
		Isotropy map equivariance is necessary in some specific cases, but we do not expect that result to extend to general cases.
		Nevertheless, using equivariant isotropy maps is still the simplest way to obtain equivariant homogeneous space integrators,
		and we are not aware of any concrete example of an equivariant integrator constructed with a non-equivariant isotropy map.
	\item
		Homogeneous spaces lacking a connection such as the space of matrix of fixed rank considered in \autoref{sec:fixedrank} might have higher order connections.
		Such connections are still local (and thus, computationally tractable), but depend on higher order derivatives of the vector field~\cite{Ve14}.
		We do not know if such a higher order equivariant connection exists for the space of fixed rank matrices.
	\item
As noted in~\cite{EdArSm98}, the representation of points on the Grassmannians as $n\times n$ matrices may be unwieldy, so there might be strategies to lift the equation on the Stiefel manifold and integrate there instead.
This is the strategy proposed in~\cite[\S\,IV.9.2]{HaLuWa06}, but we do not know if this gives equivariant integrators.
	\item 
		The retraction methods exposed in~\cite{CeOw03} on Stiefel manifolds, are not of the same type as those in this paper.
		We believe that connections may be used to conceive retraction methods of the same type on other homogeneous spaces.
		We do not know under which conditions retraction methods are equivariant.
	\item
		We showed how some homogeneous spaces have a \emph{symmetric} connection.
		In particular, in the Cartan--Schouten case, one could use the symmetric connection.
		When used along with skeletons, only the \emph{geodesics} of the connection matter, so all the connections studied in \autoref{sec:cartanschouten} will give the same methods, but there might be a way to exploit the symmetric connection.
	\item
		In \autoref{prop:scalred} we show that the connection is unique, under some assumptions.
		We don't know of any such results for the other homogeneous spaces of interest, in particular, those presented in \autoref{sec:connexamples}.
	\item
		Finally, we mention other homogeneous spaces which are used in physics, such as the Galilean and Poincaré groups, and hyperbolic spaces such as the Poincaré half space $\group{SL}(d)/\SO$, the de Sitter space, or the Lagrangian Grassmannian $\group{U}(d)/\OO$~\cite{PiTa08}, the flag manifolds~\cite[\S\,2.5]{Ki08}, not to mention the complex and quaternionic variants of the Stiefel and Grassmann manifolds.
		Note that, even if there may be theoretical results of existence of connection, in particular when the isotropy group is compact, as we saw in \autoref{prop:connexistence}, it does not mean that there is an actual, practically computable connection formula.
		Such examples are given by the isospectral manifolds \autoref{sec:isospectral} and the symmetric positive definite matrices \autoref{sec:polar}, for which the connection exists, but may be difficult to compute.
\end{itemize}

\section*{Implementation}

An implementation using the stage tree structure described in \autoref{sec:rungekutta} is available at \url{https://github.com/olivierverdier/homogint}.

\section*{Acknowledgements}

This research was supported by the Spade~Ace~Project
and by the J.C.~Kempe memorial fund.

\printbibliography

\end{document}